\documentclass[a4paper,12pt]{article}

\usepackage{amssymb,amsmath,amsthm,relsize}
\usepackage{graphicx,caption,subcaption,color,mathtools}
\usepackage[noadjust]{cite}
\usepackage[margin=2cm]{geometry}
\usepackage{enumitem}
\usepackage[all]{xy}
\renewcommand{\epsilon}{\varepsilon}

\newcommand{\R}{\mathbb{R}}
\newcommand{\Z}{\mathbb{Z}}
\newcommand{\C}{\mathbb{C}}

\renewcommand{\phi}{\varphi}

\newcommand{\cotan}{T^*}

\newtheorem{thm}{Theorem}
\newtheorem{lemma}[thm]{Lemma}
\newtheorem{prop}[thm]{Proposition}
\newtheorem{cor}[thm]{Corollary}

\theoremstyle{definition}
\newtheorem{definition}[thm]{Definition}
\newtheorem{example}[thm]{Example}
\newtheorem{claim}[thm]{Claim}

\theoremstyle{remark}
\newtheorem{rmk}[thm]{Remark}

\title{On the rigidity of lagrangian products}
\author{Vinicius G. B. Ramos\footnote{Instituto de Matem\'atica Pura e
  Aplicada -- Estrada Dona Castorina, 110 -- Rio de Janeiro --
  Brazil. Email: \texttt{vgbramos@impa.br}.} \and Daniele Sepe\footnote{Universidade Federal
    Fluminense -- Instituto de Matem\'atica -- Rua Professor Marcos
    Waldemar de Freitas Reis, s/n, Bloco H, Campus do Gragoat\'a --
    Niter\'oi -- Brazil. Email: \texttt{danielesepe@id.uff.br}.}}
\date{}

\begin{document}

\maketitle

\abstract{Motivated by work of the first author, this paper studies symplectic
  embedding problems of lagrangian products that are sufficiently
  symmetric. In general, lagrangian products arise naturally in the study of
  billiards. The main result of the
  paper is the rigidity of a large class of symplectic embedding
  problems of lagrangian products in any dimension. This is achieved
  by showing that the lagrangian products under consideration are
  symplectomorphic to toric domains, and by using the Gromov width and the
  cube capacity introduced by Gutt and Hutchings to obtain rigidity.}
\section{Introduction}

The study of symplectic embeddings lies at the heart of symplectic
topology and was kickstarted by Gromov's celebrated non-squeezing
theorem \cite{gromov}. Since then, many surprising results have been
discovered highlighting the boundary between flexibility and rigidity
in symplectic topology (cf. \cite{chls,hut_sur,schlenk_survey} for
thorough overviews). For the purposes of this paper, it is important
to remark the role that symplectic capacities play in solving
symplectic embedding problems, especially in the case of
four-dimensional toric domains (cf. \cite{mdel,ccfhr,ccdan}). 

Recently, in \cite{vinicius}, the first author studied symplectic embedding problems
involving the 4-dimensional lagrangian bidisk, an example of a class
of symplectic manifolds that are known as {\em lagrangian products}
and arise naturally in the study of billiards
(cf. \cite{artost,ostrover}). The main result in \cite{vinicius} is
the computation of the optimal symplectic embeddings of the lagrangian
bidisk into a ball and an ellipsoid. The novelty of \cite{vinicius} is
to identify the lagrangian bidisk with a concave toric
domain using the standard billiard in the disk, thus allowing one to use the machinery of embedded contact homology (ECH)
capacities to solve the problem.

Inspired by \cite{vinicius}, this paper studies symplectic embedding
problems for a large class of lagrangian products in any
dimension. The main result of the paper is that, for all lagrangian
products under consideration, the corresponding symplectic embedding
problems are rigid, meaning that one cannot do better than inclusion
(see Theorem \ref{thm:main} for a precise statement). The strategy for
the proof is similar to that employed in \cite{vinicius}. Firstly, the
relevant lagrangian products are shown to be symplectomorphic to some
toric domains (see Theorem \ref{thm:sympl}). It is worthwhile
observing that these symplectomorphisms are constructed by
understanding the symplectic geometry of the billiard in the interval
(see Section \ref{sec:billiard-cube}). Secondly, the above
identification allows us to use two symplectic capacities, the well-known Gromov width and the cube
capacity recently introduced by Gutt and Hutchings in \cite{gutt_hut},
to solve the problem (see
Theorem \ref{thm:tor}). To the best of our knowledge, the results of
this paper are the first in the study of symplectic embeddings of
lagrangian products in any dimension.

The results of the present paper, as well as those of \cite{vinicius},
corroborate the connection between integrable billiards and lagrangian
products admitting an integrable Hamiltonian torus action. We plan on
investigating this relation further in future papers.

\subsection{Lagrangian products}

We start by defining the main object of study of this paper.
\begin{definition}
Given $A, B  \subset \R^n$, the {\em lagrangian product} of $A$ and
$B$, denoted by $A\times_L B$, is the following subset of $\R^{2n}$
\[A\times_L B=\left\{(x_1,y_1,\dots,x_n,y_n)\in\R^{2n} \mid (x_1,\dots,x_n)\in A\text{ and }(y_1,\dots,y_n)\in B\right\},\]
endowed with the restriction of the symplectic form $\omega_{0}=\sum_{i=1}^n dx_i\wedge dy_i$.
\end{definition}
This article studies symplectic embedding problems of lagrangian
products and the main result is that many of these embeddings problems
are rigid (see Theorem \ref{thm:main}). Inspired by \cite{vinicius}, one of the key ingredients in
the proof of the main result is to endow lagrangian products that are
`sufficiently symmetric' with
an integrable Hamiltonian toric action (see Theorem
\ref{thm:sympl}). To make the above notion precise, we introduce the
following terminology.

\begin{definition}\label{def:sym}
An open and bounded subset $A \subset \R^n$ is said to be
\begin{itemize}[leftmargin=*]
\item {\em a balanced region} if
  $(x_1,\dots,x_n)\in A\Rightarrow
  \left[-|x_1|,|x_1|\right]\times\dots\times\left[-|x_n|,|x_n|\right]\subset
  A$;
\item {\em a symmetric region} if it is balanced and invariant under permutation of any two coordinates.
\end{itemize}
A balanced or symmetric region $A$ is {\em convex} if $A
\subset \R^n$ is a convex subset, while it is {\em concave} if
$\R^n_{\geq 0}\smallsetminus A$ is a convex subset of $\R^n$ (see Figure
\ref{fig:tor} (a) and (b)).
\end{definition}

\begin{example}\label{exm:unit_ball}
  For any $n \geq 1$ and any $p \in \left[1,\infty\right]$, the open unit
  ball in the $L^p$-norm in $\R^n$, denoted by $B^n_p$,
  is a symmetric region.
\end{example}
Given $X_1,X_2\subset\R^{2n}$, we say that {\em $X_1$ symplectically
  embeds into $X_2$} if there exists a smooth embedding from $X_1$
into $X_2$ preserving $\omega_0$. If $X_1$ symplectically embeds in
$X_2$, we write $X_1\hookrightarrow X_2$. The following theorem is the
main result of this paper.
\begin{thm}\label{thm:main}
Let $A$ and $A'$ be subsets of $\R^n$ satisfying one of the conditions below.
\begin{enumerate}[label=(\roman*)]
\item $A\in\{B_1^n,B_\infty^n\}$ and $A'$ is a convex or concave balanced region,
\item $A$ is a convex symmetric region, and
  $A'\in\{B_1^n,B_\infty^n\}$,
\item $A$ is a convex symmetric region, and $A'$ is a concave
  symmetric region,
\item $A=B^n_p$ and $A'=r\cdot B^n_q$ for some $p,q\in[1,\infty]$ and $r\in]0,\infty[$.
\end{enumerate}
Then
\[B_\infty^n\times_L A \hookrightarrow  B_\infty^n \times_L A'\iff A\subset A'.\]
\end{thm}

\begin{rmk}\label{rmk:propaganda}
  To the best of our knowledge, Theorem \ref{thm:main} is one of very
  few symplectic embedding results in dimensions greater than four,
  particularly for (families of) bounded sets. Some results in higher
  dimensions can be found in
  \cite{guth,hind_kerman,buse_hind,cg_hind}. 
\end{rmk}

The proof of Theorem \ref{thm:main} goes in two steps. First, we prove
the existence of a symplectomorphism between any lagrangian product of
the form $B^n_{\infty} \times_L A$, where $A$ is balanced, and an
appropriate toric domain (see Definition \ref{defn:toric_domain} and
Theorem \ref{thm:sympl}). This allows to reformulate Theorem
\ref{thm:main} in
terms of symplectic embeddings between certain toric domains and their
moment map images (see
Theorem \ref{thm:tor}). To solve the latter problem, we use two
symplectic capacities to
show that we cannot do better than inclusion in the corresponding
cases: the Gromov width and the cube capacity. The latter was recently
introduced in \cite{gutt_hut}.

\begin{figure}
\centering
\begin{subfigure}[b]{0.4\textwidth}
\centering
\def\svgwidth{0.5\textwidth}
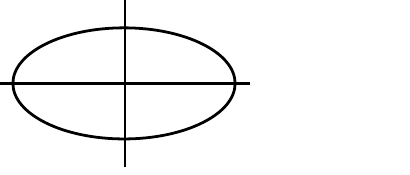
\caption{Convex balanced region}
\label{fig:conv_reg}
\end{subfigure}
\quad
\begin{subfigure}[b]{0.4\textwidth}
\centering
\def\svgwidth{0.4\textwidth}
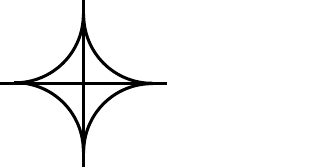
\caption{Concave symmetric region}
\label{fig:conc_reg}
\end{subfigure}\\

\vspace{0.35cm}
\begin{subfigure}[b]{0.4\textwidth}
  \centering
\def\svgwidth{0.7\textwidth}
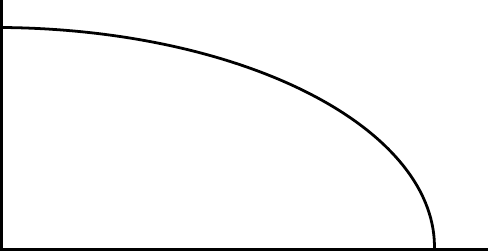
\caption{Convex toric domain}
\label{fig:conv_tor}
\end{subfigure}
\quad
\begin{subfigure}[b]{0.4\textwidth}
\centering
\def\svgwidth{0.5\textwidth}
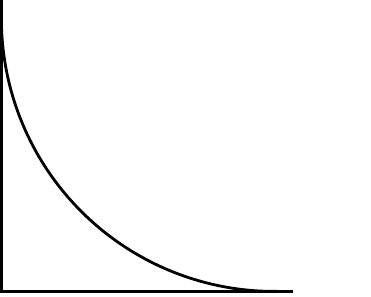
\caption{Concave symmetric toric domain}
\label{fig:conc_tor}
\end{subfigure}
\caption{Balanced regions and their corresponding toric domains}
\label{fig:tor}
\end{figure}

\subsection{Toric domains}\label{sec:tor}
Consider the standard integrable toric action
$\mathbb{T}^n = \R^n/\Z^n \curvearrowright \left(\R^{2n} = \C^n,\omega_0\right)$ defined by
$$ \left(\theta_1,\dots,\theta_n\right) \cdot \left(z_1,\dots,z_n\right) =
\left(e^{2\pi i \theta_1}z_1,\dots,e^{2\pi i \theta_n}z_n\right). $$
\noindent
Identifying the dual of the Lie algebra of $\mathbb{T}^n$ with $\R^n$,
one of the moment maps for the above action is $\boldsymbol{\mu}(z_1,\dots,z_n) =
\left(\pi \lvert z_1 \rvert^2, \dots \pi \lvert z_n\rvert^2\right)$.

\begin{definition}\label{defn:toric_domain}
  Given an open subset $\Omega \subset \R^n_{\geq 0}$,
  the {\em toric domain associated to $\Omega$} is the symplectic manifold
  $\left(X_\Omega,\omega_\Omega\right)$, where $X_{\Omega}:= \boldsymbol{\mu}^{-1}\left(\Omega\right)$
  and $\omega_\Omega = \omega_{0}|_{X_\Omega}$.
\end{definition}
Henceforth, to simplify notation, we denote the toric domain
associated to $\Omega$ by $X_\Omega$. \\

Given a balanced region $A \subset \R^n$, we set $|A|:=A\cap
\R^n_{\geq 0}$. We note that, since $A$ is balanced, $A$ is determined
by $|A|$; moreover, $|A| \subset \R^n_{\geq 0}$ is open. For any
subset $U \subset \R^n$, we set
$$4U:=\left\{\left(x_1,\ldots,x_n\right) \in \R^n \, \bigg| \,
  \left(\frac{1}{4}x_1,\ldots,\frac{1}{4}x_n\right) \in U\right\}.$$
\noindent
The following result is the first step towards proving Theorem \ref{thm:main}.
\begin{thm}\label{thm:sympl}
Let $A \subset \R^n$ be a balanced region. Then there is a symplectomorphism
\[B_\infty^n\times_L A \cong X_{4|A|}.\]
\end{thm}

Figures \ref{fig:tor} (c) and (d) show the moment map images of the
toric domains obtained from the regions in Figures \ref{fig:tor} (a)
and (b) respectively. The proof of Theorem \ref{thm:sympl}, carried
out in Section \ref{sec:billiard-cube}, uses the integrability of a system which models
billiards on an interval and the fact that we decompose
$B^n_\infty\times_L\R^n$ as a product of $n$ symplectic factors (see
Section \ref{sec:billiard-cube}). This is the main novelty
of this paper and might be of independent interest. In
spirit, it is a similar result to the existence of a symplectomorphism
between the lagrangian bidisk and a concave toric domain proved in
\cite{vinicius}, although the integrable system in \cite{vinicius} is
different from the one in the current paper.

\begin{rmk}\label{rmk:context}
  Some particular cases of Theorem \ref{thm:sympl} are known or could
  be easily deduced from existing results in the
  literature. For instance, the ideas in \cite[Section 2]{schlenk}
  allow to prove Theorem \ref{thm:sympl} in the case in which $A$ is a
  parallelepiped. Besides this family, little seems to be known
  in general, although it is worth mentioning that, if $n=2$,
  \cite[Corollary 4.2]{lat_mc_schl} proves the case $A = B^2_1$
  by using a non-trivial result due to McDuff (cf. \cite[Theorem
  1.1]{mcduff_bl}), which is intrinsically different from our
  constructive techniques and only applicable in this specific case.
\end{rmk}

Assuming Theorem \ref{thm:sympl}, Theorem \ref{thm:main} can be
restated in terms of toric domains. To this end, we introduce the
following notion.
\begin{definition}\label{defn:con_con_toric}
Given a convex (respectively concave) balanced region $A\subset \R^n$, $X_{|A|}$ is said to be a
{\em convex} (respectively {\em concave}) toric domain. If, in
addition, $A$ is symmetric, $X_{|A|}$ is said to be {\em symmetric}.
\end{definition}
\begin{rmk}
If $\Omega \subset \R^n_{\ge 0}$ is open and bounded, the existence of a balanced region $A$ such that $\Omega=|A|$ is equivalent to the following condition:
\begin{equation}\label{eq:cond}(x_1,\dots,x_n)\in\Omega\Rightarrow
  [0,x_1]\times \dots \times [0,x_n]\subset \Omega.\end{equation}
In particular, the notions of convex and concave toric domains from
\cite{ccfhr,hut_bey} coincide with those of Definition
\ref{defn:con_con_toric}, except that we consider open domains instead
of compact domains. We note that the definition of convex toric domain
in \cite{ccdan} is slightly different and allows for toric domains that do not satisfy \eqref{eq:cond}.
\end{rmk}

For any $n \geq 1$ and any $p \in \left[1,\infty\right]$, we set $\Omega^n_p :=|B^n_p|$. Assuming Theorem \ref{thm:sympl}, Theorem \ref{thm:main} is a straightforward consequence of the following result.
\begin{thm}\label{thm:tor}
  Let $\Omega$ and $\Omega'$ be open subsets of $\R^n_{\ge 0}$
  satisfying one of the conditions below.
  \begin{enumerate}[label=(\roman*), ref = (\roman*)]
  \item \label{item:5}$\Omega\in\{\Omega^n_1,\Omega^n_\infty\}$ and $X_{\Omega'}$ is a convex or concave toric domain,
  \item \label{item:8} $X_\Omega$ is a convex symmetric toric domain and $\Omega'\in\{\Omega^n_1,\Omega^n_\infty\}$,
  \item \label{item:9} $X_\Omega$ is a convex symmetric toric domain and $X_{\Omega'}$ is a concave symmetric toric domain,
  \item \label{item:10} $\Omega=\Omega^n_p$ and $\Omega'=r\cdot \Omega^n_q$ for some $p,q\in[1,\infty]$ and $r\in]0,\infty[$.
  \end{enumerate}
  Then
\begin{align*}
X_\Omega \hookrightarrow X_{\Omega'} &\iff \Omega\subset \Omega',\\
\end{align*}
\end{thm}
\begin{rmk}\label{rmk:ball_poly}
The domains $X_{\Omega^n_1}$ and $X_{\Omega^n_\infty}$ are usually known as the ball $B(1)=E(1,\dots,1)$ and the polydisk $P(1,\dots,1)$, respectively.
\end{rmk}

\subsection{Symplectic capacities}\label{sec:sympl-capac}
The proof of Theorem \ref{thm:tor} provided below uses symplectic capacities. A symplectic capacity is a map $c$ from a certain class of symplectic manifolds to $[0,\infty]$ with the following properties:
\begin{itemize}
\item[(a)] $c(X,r \cdot \omega)=r\cdot c(X,\omega)$, for $r\in]0,\infty[$;
\item[(b)] $(X_1,\omega_1)\hookrightarrow (X_2,\omega_2)\Rightarrow c(X_1,\omega_1)\le c(X_2,\omega_2)$.
\end{itemize}
For star-shaped
domains $X \subset \R^{2n}$ equipped with the standard symplectic
form\footnote{Since we only use capacities of subsets of $\R^{2n}$
  equipped with the standard symplectic form, we drop the
  symplectic form from the notation.}, the following quantities are
symplectic capacities:
\begin{equation*}
  \begin{split}
    c_{1}(X) &=\sup\left\{r\in\R \mid X_{r\cdot \Omega^n_1}
      \hookrightarrow X \right\}, \\
    c_{\infty}(X) &=\sup \left\{r\in\R \mid X_{r\cdot \Omega^n_\infty}
      \hookrightarrow X \right\}.
\end{split}
\end{equation*}

\begin{rmk}\label{rmk:sc}
  The capacity $c_1(X)$ was first introduced by Gromov in
  \cite{gromov} and is known in the literature as the Gromov width of
  $X$, while $c_{\infty}(X)$ is shown to be a capacity by Gutt and Hutchings in
  \cite{gutt_hut} and is the analog of $c_1(X)$ for a cube.
\end{rmk}

The following result is a simple consequence of some results in \cite{gutt_hut}.
\begin{thm}\label{thm:c1cb}
Let $X_\Omega$ be a convex or concave toric domain. Then
\begin{align}
c_1(X_\Omega)&=\max\{r\in\R \mid r\cdot \Omega^n_1\subset\Omega\},\label{eq:c1}\\
c_\infty(X_\Omega)&=\max\{r\in\R \mid r\cdot \Omega^n_\infty\subset \Omega\}.\label{eq:cb}
\end{align}
\end{thm}
\begin{proof}
In \cite{gutt_hut} Gutt and Hutchings define a normalized symplectic
capacity $c_1^{SH}$ for star-shaped domains in $\R^{2n}$, {\it i.e.}
on balls and cylinders, $c_1^{SH}$ agrees with $c_1$. In particular,
for any star-shaped domain $X \subset \R^{2n}$,
\begin{equation}
  \label{eq:11}
  c_1(X)\le c_1^{SH}(X).
\end{equation}
\noindent
Moreover, they show in \cite[Theorem 1.6]{gutt_hut} that for a convex toric domain $X_\Omega$,
\[c_1^{SH}(X_\Omega)=\min_{i=1,\dots,n}\sup\{r\in\R \mid r\cdot
e_i\in\Omega\}.\] From the convexity of $\Omega$ and the definition of
$c_1$, we obtain
\begin{equation}\label{eq:sh}
c_1^{SH}(X_\Omega)=\max\{r\in\R \mid r\cdot \Omega^n_1\subset\Omega\} \leq c_1(X_{\Omega}).
\end{equation}
\noindent
Combining \eqref{eq:11} and \eqref{eq:sh}, we obtain \eqref{eq:c1}
for a convex toric domain. For a concave toric domain, \eqref{eq:c1} is a direct consequence of \cite[Corollary 1.16]{gutt_hut}.

To complete the proof, observe that \cite[Theorem 1.18]{gutt_hut} gives \eqref{eq:cb} for convex and concave toric domains.
\end{proof}

The capacities $c_1$ and $c_{\infty}$ can be used to prove Theorem
\ref{thm:tor}, which, in turn, provides a proof of the main result of
this paper assuming Theorem \ref{thm:sympl}.
\begin{proof}[Proof of Theorem \ref{thm:tor}]
  Given open subsets $\Omega, \Omega' \subset \R^n_{\geq 0}$, the
  inclusion $\Omega \subset \Omega'$ implies the existence of a
  symplectic embedding $X_{\Omega} \hookrightarrow X_{\Omega'}$ (without
  imposing any restrictions). Therefore, it remains to prove the other
  implication. To this end, suppose that $X_\Omega\hookrightarrow X_{\Omega'}$ where $\Omega$
  and $\Omega'$ satisfy one of the conditions \ref{item:5} --
  \ref{item:10}. The aim is to show that $\Omega\subset\Omega'$. We proceed case by case.
  \begin{enumerate}[label=(\roman*), ref=(\roman*)]
  \item \label{item:11} This is a direct consequence of Theorem \ref{thm:c1cb}.
    First consider the case $\Omega=\Omega^n_1$. It follows from
    \eqref{eq:c1} that $c_1(X_{\Omega^n_1})=1$. So
    $1=c_1(X_{\Omega^n_1})\le c_1(X_{\Omega'})$. Again from
    \eqref{eq:c1} we obtain $\Omega^n_1\subset \Omega'$. The case
    $\Omega=\Omega^n_\infty$ is dealt with analogously
    using $c_\infty$ and \eqref{eq:cb}.

  \item \label{item:12} Suppose first that $\Omega'=\Omega^n_1$. It follows from
    \eqref{eq:cb} that $1/n=c_{\infty}(X_{\Omega^n_1})\ge
    c_{\infty}(X_\Omega)$. Since $\Omega$ is symmetric and convex, it
    follows from \eqref{eq:cb} that $\Omega$ lies below the hyperplane
    normal to $(1,\dots,1)$ passing through the point
    $(c_{\infty}(X_\Omega),\dots,c_{\infty}(X_\Omega))$. So
    $x_1+\dots+x_n< n\cdot c_{\infty}(X_\Omega)\le 1$ for all
    $(x_1,\dots,x_n)\in\Omega$. Therefore $\Omega\subset\Omega^n_1$.

    On the other hand, suppose that $\Omega'=\Omega^n_\infty$. If
    $\{e_1,\dots,e_n\}$ denotes the canonical basis of $\R^n$,
    $\Omega$ being symmetric implies that for all $i,j = 1,\ldots, n$,
    $$ \sup\{r>0 \mid r\cdot e_i\in\Omega\}  = \sup\{r>0 \mid r\cdot
    e_j\in\Omega\}.$$
    \noindent
    In particular, since $\Omega$ is convex, it follows from \eqref{eq:c1}
    that, for any $i=1,\ldots, n$, $c_1(X_\Omega)=\sup\{r>0 \mid r\cdot e_i\in\Omega\}$.
    Therefore, $\Omega\subset [0,c_1(X_\Omega)]^n$; since
    $c_1(X_\Omega)\le c_1(X_{\Omega^n_\infty})=1$, it follows that
    $\Omega\subset[0,1]^n$. As $\Omega$ is open, $\Omega \subset
    [0,1[^n = \Omega^n_\infty$ as desired.

  \item \label{item:13} Arguing as in the first part of
    \ref{item:12}, it follows that if $X_{\Omega'}$ is concave
    and symmetric, then $c_{\infty}(X_{\Omega'})\cdot\Omega^n_1\subset
    \Omega'$. Since $X_\Omega\hookrightarrow X_{\Omega'}$, it
    follows from \ref{item:12} that
    \[\Omega\subset c_{\infty}(X_\Omega)\cdot \Omega^n_1\subset
    c_{\infty}(X_{\Omega'})\cdot \Omega^n_1\subset \Omega'.\]

  \item \label{item:14} Suppose first that $p\le q$. It follows from \eqref{eq:c1}
    that $c_1(\Omega^n_p)=c_1(\Omega^n_q)=1$. So
    $\Omega^n_p\hookrightarrow r\cdot \Omega^n_q$ implies that $1\le
    r$. Since $p\le q$, we conclude that
    $\Omega^n_p\subset\Omega^n_q\subset r\cdot \Omega^n_q$.

    Suppose that $q\le p$. From \eqref{eq:cb} we obtain $c_{\infty}(X_{\Omega^n_s})=\frac{1}{n^{1/s}}$. So
    \begin{equation}\label{eq:pqr}\frac{1}{n^{1/p}}\le \frac{r}{n^{1/q}}.
    \end{equation}
    Let $(x_1,\dots,x_n)\in\Omega^n_p$. It follows from H\"{o}lder's inequality and \eqref{eq:pqr} that
    \[ \sum_{i=1}^n \left(\frac{x_i}{r}\right)^q\le \frac{1}{r^q}\left(\sum_{i=1}^n x_i^p\right)^{\frac{q}{p}} n^{\frac{p-q}{p}}\le \left(\frac{n^{\frac{p-q}{pq}}}{r}\right)^q\le 1,\]
    \noindent
    which implies that $(x_1,\dots,x_n)\in r\cdot \Omega^n_q$.
  \end{enumerate}
\end{proof}

\subsection{Outline of the paper}
The rest of this paper is structured as follows. In Section
\ref{sec:emb}, we give another application of Theorems \ref{thm:main}
and  \ref{thm:sympl}
to symplectic embeddings in dimension 4. Section
\ref{sec:billiard-cube} contains the proof of Theorem \ref{thm:sympl},
which relies on understanding a family of integrable systems modeling
the billiard in the interval.

\section{Symplectic embeddings}\label{sec:emb}
In this section, we discuss the rigidity of some symplectic embeddings
involving the lagrangian bidisk studied in \cite{vinicius}, proving
that both rigidity and flexibility occur (see Theorem \ref{thm:balls}
and Corollary \ref{cor:three}). First we introduce an equivalence relation for four-dimensional lagrangian products.
\begin{definition}\label{def:equiv}
Let $A_1$, $B_1$, $A_2$ and $B_2$ be open sets of $\R^2$ containing
the origin. The lagrangian products $A_1\times_L B_1 $ and
$A_2\times_L B_2$ are {\em equivalent} if there exist $a>0$ and $U\in
SO(2)$ such that $A_1=a U\cdot A_2$ and $B_1=a^{-1} U\cdot B_2$, or
$B_1=a U\cdot A_2$ and $A_1=a^{-1} U\cdot B_2$. In this case, we write $A_1\times_L B_1\sim A_2\times_L B_2$
\end{definition}
Observe that two equivalent lagrangian products are symplectomorphic.

\begin{definition}
Let $A$, $B$, $C$ and $D$ be connected, open sets of $\R^2$ containing
the origin. The symplectic embedding problem $A\times_L
B\overset{?}{\hookrightarrow} C\times_L D$ is {\em rigid}
if \[A\times_L B\hookrightarrow (a C)\times_L D \text{ for some } a >
0 \,\Rightarrow A\times_L B\sim A'\times_L B'\subset (a C')\times_L
D'\sim C\times_L D,\]
\noindent
for some open subsets $A',B',C',D'$ of $\R^2$ containing the origin.
\end{definition}

The following theorem is the main result of this section.

\begin{thm}\label{thm:balls}
  For any $p\in[2,+\infty]$ the symplectic embedding problems
  \begin{equation*}
    B^2_2\times_L B^2_2 \overset{?}{\hookrightarrow} B^2_\infty\times_L
    B^2_p \quad \text{and} \quad B^2_\infty\times_L B^2_p\overset{?}{\hookrightarrow} B^2_2\times_L B^2_2
  \end{equation*}
  are rigid.
\end{thm}
Combining Theorems \ref{thm:main} and \ref{thm:balls}, we obtain the
following result.
\begin{cor}\label{cor:three}
For any $p,q,r,s\in\{1,2,\infty\}$
with
\begin{equation}
\label{eq:notin}
(p,q,r,s)\not\in\{(1,\infty,2,2),(\infty,1,2,2),(2,2,1,\infty),(2,2,\infty,1)\},
\end{equation}
the symplectic embedding problem $B^2_{p}\times_L B^2_{q}\overset{?}{\hookrightarrow} B^2_{r}\times_L B^2_{s}$ is rigid.
\end{cor}

In fact, Corollary \ref{cor:three} is optimal, in the sense that if
\eqref{eq:notin} does not hold, then $B^2_{p}\times_L
B^2_{q}\overset{?}{\hookrightarrow} B^2_{r}\times_L B^2_{s}$ is not
rigid (see Section \ref{sec:diam-disks-squar}).

\subsection{The lagrangian bidisk}\label{sec:lagrangian-bidisk}
The lagrangian bidisk $B^2_2 \times_L B^2_2$ is the only lagrangian
product that appears in Theorem \ref{thm:balls} and not in Theorem
\ref{thm:main}. While the techniques of the present paper do not allow
to identify the lagrangian bidisk with a toric domain, the first
author proved in \cite{vinicius} that $B^2_2 \times_L B^2_2$ can be
endowed with an effective Hamiltonian $\mathbb{T}^2$-action. This is
the content of the following result, stated below without proof.
\begin{thm}[{\cite[Theorem 3]{vinicius}}]\label{thm:bidisk}
  Let $\Omega_0\subset \R^n_{\ge 0}$ be the open subset of bounded by the coordinate axes and the curve parametrized by
  \[\gamma(\alpha)=2\left(\sin\alpha-\alpha\cos\alpha,\sin\alpha+(\pi-\alpha)\cos\alpha\right),\quad \alpha\in[0,\pi].\]
Then $B^2_2\times_L B^2_2$ is symplectomorphic to the toric domain $X_{\Omega_0}$.
\end{thm}

\begin{proof}[Proof of Theorem \ref{thm:balls}]
Fix $p\in[2,\infty[$. It follows from Theorem \ref{thm:bidisk} and
formulae \eqref{eq:c1} and \eqref{eq:cb} that
\begin{align}
  c_1\left(B^2_2\times_L B^2_2 \right)&= c_1(X_{\Omega_0})=4,\label{eq:c1concave}\\
  c_\infty\left(B^2_2\times_L B^2_2 \right)&= c_{\infty}(X_{\Omega_0})=2.\label{eq:cbconcave}
\end{align}
\noindent
Suppose first that $B^2_2\times_L B^2_2\hookrightarrow B^2_\infty \times_L
a B^2_p$ for some $a> 0$. 
It follows from Theorems \ref{thm:sympl} and \ref{thm:bidisk} that $X_{\Omega_0}\hookrightarrow X_{4a\cdot\Omega^2_p}$.
From \eqref{eq:c1} and \eqref{eq:c1concave} we obtain
\[4=c_1(X_{\Omega_0})\le 4a \cdot c_1(X_{\Omega^2_p})=4a.\]
So $a\ge 1$. Since $B^2_2\subset B^2_p\subset B^2_{\infty}$, it
follows that $B^2_2\times_L B^2_2 \subset B^2_\infty\times_L
B^2_p$. Therefore $B^2_2\times_L B^2_2 \overset{?}{\hookrightarrow}
B^2_\infty\times_L B^2_p$ is rigid.

Suppose that $B^2_\infty \times_L B^2_p\hookrightarrow a B^2_2\times_L
B^2_2$ for some $a >0$, so that $X_{4\cdot \Omega^2_p}\hookrightarrow X_{a\cdot\Omega_0}$. From \eqref{eq:cb} and \eqref{eq:cbconcave} we obtain
\[\frac{4}{2^{1/p}}=c_{\infty}(X_{\Omega^2_p})\le a\cdot c_{\infty}(X_{\Omega_0})=2a.\]
So $a\ge \frac{2}{2^{1/p}}$. It follows from a simple calculation that \[B^2_p\subset \frac{2^{1/2}}{2^{1/p}}B^2_2\quad\text{and}\quad B^2_\infty\subset 2^{1/2} B^2_2.\] Hence
\[B^2_\infty \times_L B^2_p\subset \left(\frac{2^{1/2}}{2^{1/p}}B^2_2\right)\times_L \left(2^{1/2} B^2_2\right)\sim \frac{2}{2^{1/p}}B^2_2\times_L\times B^2_2.\]
Therefore $B^2_p\times_L B^2_\infty\overset{?}{\hookrightarrow}  B^2_2\times_L B^2_2$ is rigid.

The case $p = \infty$ can be dealt with analogously by substituting
$\Omega^2_p$ by $\Omega^2_\infty$ and $1/p$ by $0$ in the above calculations.
\end{proof}

\subsection{Diamonds, disks and squares}\label{sec:diam-disks-squar}
The aim of this section is to prove Corollary \ref{cor:three} and
explain why it is optimal.
\begin{proof}[Proof of Corollary \ref{cor:three}]
Begin by observing that there exist $a > 0$ and $U\in SO(2)$ such that
$B^2_1= aU\cdot B^2_\infty$, and that $B^2_2$ is $SO(2)$-invariant.
So
\begin{equation}\label{eq:sim1}
B^2_2\times_L B^2_1\sim B^2_\infty \times_L a B^2_2\sim a B^2_\infty\times_L B^2_2\sim B^2_1\times_L B^2_2,\end{equation}
Moreover
\begin{equation}\label{eq:sim2}
B^2_1\times_L B^2_1\sim a^2 B^2_\infty\times_L B^2_\infty\quad\text{and}\quad B^2_\infty \times_L B^2_1\sim B^2_\infty\times_L B^2_1.\end{equation}

Fix $p,q,r,s\in\{1,2,\infty\}$ satisfying \eqref{eq:notin}. It follows from \eqref{eq:sim1} and \eqref{eq:sim2} that the $B^2_p\times_L B^2_q\overset{?}{\hookrightarrow}B^2_r\times_L B^2_s$ is equivalent to one of the embedding problems considered either in Theorem \ref{thm:main} or in Theorem \ref{thm:balls}.
\end{proof}

It remains to show that the symplectic embedding problems
\[B^2_\infty \times_L B^2_1\overset{?}{\hookrightarrow} B^2_2\times_L B^2_2\quad\text{and}\quad B^2_2\times_L B^2_2\overset{?}{\hookrightarrow} B^2_\infty \times_L B^2_1\]
are not rigid. It follows from a simple calculation that $4\Omega^2_1\subset \Omega_0$. So $B^2_\infty \times_L B^2_1\hookrightarrow B^2_2\times_L B^2_2$. However, if $B^2_1\subset a B^2_2$ and $B^2_\infty\subset b B^2_2$, then $a\ge 1$ and $b\ge\sqrt{2}$. So $ab\ge \sqrt{2}>1$. Therefore the embedding problem $B^2_\infty \times_L B^2_1\overset{?}{\hookrightarrow} B^2_2\times_L B^2_2$ is not rigid.

On the other hand, it is shown in \cite{vinicius} that $X_{\Omega_0}\hookrightarrow X_{3\sqrt{3}\cdot \Omega^2_1}$. So $B^2_2\times_L B^2_2\hookrightarrow B^2_\infty \times_L \frac{3\sqrt{3}}{4} B^2_1$. However, if $B^2_2\subset a B^2_1$ and $B^2_2\subset b B^2_\infty$, then $a\ge\sqrt{2}$ and $b\ge 1$ and hence \[ab\ge\sqrt{2}>\frac{3\sqrt{3}}{4}.\]
Therefore the embedding problem $B^2_2\times_L B^2_2\hookrightarrow B^2_1\times_L B^2_\infty$ is not rigid.

\begin{rmk}
 We could say that an embedding problem of toric domains $X_{\Omega}\overset{?}{\hookrightarrow}X_{\Omega'}$ is rigid if
 \[X_{\Omega}\hookrightarrow X_{a\cdot\Omega'}\Rightarrow \Omega\subset a\cdot\Omega'.\]
Based on the calculations above, $B^2_\infty \times_L B^2_1\overset{?}{\hookrightarrow} B^2_2\times_L B^2_2$ is not rigid as an embedding problem of lagrangian products, but it is rigid as an embedding problem of toric domains. However, $B^2_2\times_L B^2_2\overset{?}{\hookrightarrow} B^2_\infty\times_L B^2_1$ is not rigid in either sense.
\end{rmk}

\section{From balanced regions to toric domains}\label{sec:billiard-cube}
The aim of this section is to prove Theorem \ref{thm:sympl}, thus
completing the proof of the main result of the paper, Theorem
\ref{thm:main}. Our strategy to prove Theorem \ref{thm:sympl} is
inspired by some of the arguments in \cite[Section 2]{vinicius} and
can be broken down in the following three steps:
\begin{enumerate}[label = (\roman*), ref= (\roman*), leftmargin=*]
\item \label{item:15} For any $n \geq 1$ and any $\epsilon > 0$, we
  define an integrable system $\Phi_{\epsilon}: B^n_{\infty}
  \times_L \R^n \to \R^n$ related to $n$ uncoupled billiards in the
  interval. We prove that, for any $\epsilon >0$, $\Phi_{\epsilon}: B^n_{\infty}
  \times_L \R^n \to \R^n$ is isomorphic to $\boldsymbol{\mu} : \R^{2n}
  \to \R^n$, the integrable system obtained by considering (one of) the moment
  map(s) of the standard Hamiltonian $\mathbb{T}^n$-action on $\R^{2n}$ (see Section
  \ref{sec:tor}). If
  $\left(\boldsymbol{\Psi}_{\epsilon},\mathbf{I}_{\epsilon}\right)$
  denotes the above isomorphism for a fixed $\epsilon$, we also show
  that, in some sense, the maps $\mathbf{I}_{\epsilon}$ possess a limit as $\epsilon
  \to 0$, which we denote by $\mathbf{I}_0$. 
\item \label{item:16} Fix $n \geq 1$ and a balanced region $A \subset \R^n$. Using
  the family of isomorphisms of integrable systems
  $\left(\boldsymbol{\Psi}_{\epsilon},\mathbf{I}_{\epsilon}\right)$
  and the map $\mathbf{I}_0$ of \ref{item:15}, we construct a family of nested
  symplectic submanifolds
  of $B^n_{\infty}
  \times_L A$ parametrized by $\epsilon$ whose images under
  $\boldsymbol{\Psi}_{\epsilon}$ are a nested family of symplectic
  submanifolds exhausting $X_{4|A|}$ (see Lemma
  \ref{lemma:family} for a precise statement).
\item \label{item:17} Using the symplectic isotopy extension theorem
  (cf. \cite[Proposition 4]{auroux} and \cite{banyaga}) and the compact
  exhaustions of \ref{item:16}, we construct the desired
  symplectomorphism between $B^n_\infty \times_L A$ and $X_{4|A|}$.
\end{enumerate}

The structure of this section is as follows. Section
\ref{sec:family-integr-syst} constructs the desired family of integrable
systems on $B^n_{\infty} \times_L \R^n$, while Section
\ref{sec:constr-sympl} deals with steps \ref{item:16} and
\ref{item:17}, thus completing the proof of Theorem \ref{thm:sympl}
and, hence, of the main result, Theorem \ref{thm:main}.

\subsection{A family of integrable systems on $B^n_{\infty}
  \times_L \R^n$}\label{sec:family-integr-syst}

\subsubsection{The category of integrable systems}\label{sec:categ-integr-syst}
Before constructing the desired family of integrable systems on $B^n
\times_L \R^n$, we recall some basic notions pertaining to integrable systems that are used throughout the
paper. 

\begin{definition}\label{defn:is}
  An {\em integrable
    system} on a $2n$-dimensional symplectic manifold
  $\left(M,\omega\right)$ is a smooth map
  $$ \Phi:= \left(h_1,\ldots,h_n\right) : \left(M,\omega\right) \to
  \R^n$$
  \noindent
  satisfying the following conditions
  \begin{itemize}[leftmargin=*]
  \item for all $i,j = 1,\ldots, n$, $\{h_i,h_j\}=0$, where
    $\left\{\cdot,\cdot\right\}$ is the Poisson bracket on
    $C^{\infty}(M)$ induced by
    $\omega$, and
  \item the map $\Phi$ is a submersion on a dense subset of $M$.
  \end{itemize}
\end{definition}

\begin{example}\label{exm:is}
  For the purposes of this paper, the following are important examples
  of integrable systems:
  \begin{enumerate}[label = (\alph*), ref = (\alph*), leftmargin=*]
  \item \label{item:18} If $n=1$, an integrable system on a surface
    $\left(M,\omega\right)$ is a function $H \in
    C^{\infty}\left(M\right)$ whose differential does not vanish on a
    dense subset.
  \item \label{item:19} For $i=1,2$, let $\Phi_i : \left(M_i,\omega_i\right) \to
    \R^{n_i}$ be an integrable system. Then the map $\Phi :
    =\left(\Phi_1,\Phi_2\right) : \left(M_1 \times M_2,\omega_1 \oplus
      \omega_2\right) \to \R^{n_1 + n_2}$ is an integrable system,
    where $\omega_1 \oplus \omega_2 = \mathrm{pr}_1^*
    \omega_1 + \mathrm{pr}^*_2\omega_2$ and, for $i=1,2$, $\mathrm{pr}
    : M_1 \times M_2 \to M_i$ denotes the canonical projection.
  \item \label{item:20} A {\em symplectic toric manifold} is a triple
    $\left(M,\omega,\mu\right)$, where $\left(M,\omega\right)$ is a
    $2n$-dimensional symplectic manifold and $\mu$ is the moment map
    of an effective Hamiltonian $\mathbb{T}^n$-action on
    $\left(M,\omega\right)$. Given a symplectic toric manifold
    $\left(M,\omega\right)$ and identifying the dual of the Lie algebra
    of $\mathbb{T}^n$ with $\R^n$, the map $\mu :
    \left(M,\omega\right) \to \R^n$ defines an integrable system. (The
    fact that $\mu$ is a submersion on a dense set follows from the
    Marle-Guillemin-Sternberg local normal form for effective
    Hamiltonian actions, cf. \cite{guil_stern,marle}.) In particular,
    the following maps define integrable systems:
    \begin{itemize}[leftmargin=*]
    \item the moment map $\boldsymbol{\mu} : \R^{2n} \to \R^n$ of the
      standard Hamiltonian $\mathbb{T}^n$-action on $\R^{2n}$, and
    \item the moment map of the cotangent lift of the
      $\mathbb{T}^n$-action on $\mathbb{T}^n$ by left (or right)
      multiplication. Using the canonical trivialization $\cotan
      \mathbb{T}^n \cong \R^n \times \mathbb{T}^n$ so that the
      canonical symplectic form becomes $\sum\limits_{i=1}^n d a_i
      \wedge d\theta_i$, this moment map becomes the projection
      $\mathrm{pr}_1 : \left(\R^n \times \mathbb{T}^n,  \sum\limits_{i=1}^n d a_i
      \wedge d\theta_i\right) \to \R^n$ onto the
      first component.
    \end{itemize}
  \end{enumerate}
\end{example}

An important role in this paper is played by the following notion of
equivalence of integrable systems.

\begin{definition}\label{defn:im_is}
  Two integrable systems $\Phi_1 : \left(M_1,\omega_1\right) \to \R^{n_1}$
  and $\Phi_2 : \left(M_2,\omega_2\right) \to \R^{n_2}$ are {\em
    isomorphic} if there exists a pair $\left(\Psi, g\right)$
  consisting of a symplectomorphism
  $\Psi : \left(M_1,\omega_1\right) \to \left(M_2,\omega_2\right)$ and
  a diffeomorphism\footnote{A map $g : C \subset \R^{n_1}\to
    \R^{n_2}$ between is said to be {\em smooth} if there exists an
    open set $V$ containing $C$ and a smooth map $\tilde{g} : U \to
    \R^{n_2}$ that extends $g$.} $g :\Phi_1\left(M_1\right) \to
  \Phi_2\left(M_2\right)$ such that $\Phi_2 \circ \Psi = g \circ \Phi_1$.
\end{definition}

\begin{rmk}\label{rmk:im_prod}
  The above notion of isomorphism of integrable systems behaves well
  with respect to the product construction \ref{item:19} of Example
  \ref{exm:is}. More precisely, if, for $i=1,2$,
  $\left(\Psi_i,g_i\right)$ is an isomorphism between $\Phi_i :
  \left(M_i,\omega_i\right) \to \R^{n_i}$ and $\Phi'_i :
  \left(M'_i,\omega'_i\right) \to \R^{n'_i}$, then the pair
  $\left(\Psi_1 \times \Psi_2, g_1 \times g_2\right)$ is an
  isomorphism between $\Phi = \left(\Phi_1,\Phi_2\right) : \left(M_1 \times M_2,\omega_1 \oplus
    \omega_2\right) \to \R^{n_1 + n_2}$ and $\Phi' = \left(\Phi'_1,\Phi'_2\right) : \left(M'_1 \times M'_2,\omega'_1 \oplus
    \omega'_2\right) \to \R^{n'_1 + n'_2}$.
\end{rmk}

Another construction that is relevant for our purposes is that of
restricting integrable systems to suitable subsets.

\begin{definition}\label{defn:subsys}
  Given an integrable system $\Phi : \left(M,\omega\right) \to \R^n$
  and an open subset $U \subset M$, the {\em subsystem relative to
    $U$} is the integrable system $\Phi|_U : \left(U,\omega|_U\right)
  \to \R^n$.
\end{definition}

For any $n \geq 1$, the family of integrable systems $\Phi_{\epsilon}
: B^n_{\infty}
\times_L \R^n \to \R^n$ that we are interested in is going to be constructed
using the product construction \ref{item:19} of Example
\ref{exm:is}, since $B^n_{\infty} \times_L \R^n$ is symplectomorphic to
the (symplectic) product of $n$ copies of $B^1_{\infty} \times_L
\R$. Thus, firstly we define the relevant family of integrable systems
and prove all desired properties in the case $n=1$ (see Section
\ref{sec:billiard-interval}), and, secondly, we consider the general
case (see Section \ref{sec:general-case}).

\subsubsection{The one dimensional case}\label{sec:billiard-interval}
For any $\epsilon > 0$, consider the integrable system
$H_{\epsilon}: B^1_\infty \times_L \R \to \R$, where $H_{\epsilon}(x,y) = \frac{1}{2}\left(y^2 + \epsilon
  \frac{1}{1-x^2}\right)$ and let $x,y$ denote canonical coordinates on
$\R^2$.

\begin{rmk}\label{rmk:relation_billiard}
  The family of integrable systems $\left\{H_{\epsilon}: B^1_\infty
    \times_L \R \to \R\right\}_{\epsilon > 0}$ is related to the
  dynamics of the billiard in the interval
  $\left[-1,+1\right]$ as follows. Firstly, observe that $
  B^1_{\infty} \times_L \R$ is symplectomorphic
  to $\left(\cotan B^1_{\infty},\omega_{\mathrm{can}}\right)$. Secondly,
  the potential $V(x)=\frac{1}{2\left(1-x^2\right)}$ satisfies the properties to fit in the approximation scheme first introduced
  in \cite{benci_gian} that allow us to study the billiard in the
  interval $\left[-1,+1\right]$ as a limit of Hamiltonian systems on
  the cotangent bundle of $B^1_{\infty}$.
\end{rmk}

The
following result, stated below without proof,
establishes some basic properties of $H_{\epsilon}:B^1_{\infty}
\times_L \R \to \R$ for any $\epsilon
> 0$.

\begin{prop}\label{prop:basic_is}
  For any $\epsilon > 0$,
  \begin{enumerate}[label = (\arabic*), ref = (\arabic*), leftmargin=*]
  \item \label{item:1} the image of $H_{\epsilon}$ equals
    $\left[\frac{\epsilon}{2},+\infty \right[$;
  \item \label{item:2} the only singular point of $H_{\epsilon}$ is $(0,0)$, which
    equals the fiber
    $H^{-1}_{\epsilon}\left(\frac{\epsilon}{2}\right)$;
  \item \label{item:4} the Hessian of $H_{\epsilon}$ at $(0,0)$ is
    positive definite;
  \item \label{item:3} the map $H_{\epsilon} : B^1_{\infty} \times_L \R \to \R$ is proper.
  \end{enumerate}
\end{prop}

\begin{rmk}
  Property \ref{item:4} is equivalent to stating that, for any
  $\epsilon > 0$, the point $(0,0)$ is a {\em non-degenerate} singular
  point of {\em elliptic type}
  for the integrable system $H_{\epsilon}: B^1_{\infty} \times_L \R \to \R$,
  cf. \cite[Introduction]{eliasson} and \cite[Section I.3]{duf_mol}
  for more details.
\end{rmk}

An immediate consequence of Proposition \ref{prop:basic_is} is the
following simple, yet useful, result.

\begin{cor}\label{cor:compact_connected}
  For any $\epsilon > 0$, the fibers of $H_{\epsilon}$ are compact and connected.
\end{cor}
\begin{proof}
  Fix $\epsilon > 0$. Property \ref{item:3} implies that the fibers of
  $H_{\epsilon}$ are compact. Using property \ref{item:2}, it remains
  to check that the fiber $H^{-1}_{\epsilon}(c)$, for $c \in \left]
    \frac{\epsilon}{2}, + \infty \right[$, is connected. To this end,
  consider the restriction of $H_{\epsilon}$ to $\left(B^1_{\infty} \times_L \R\right)
  \smallsetminus \{(0,0)\}$ as a map onto $\left]
    \frac{\epsilon}{2}, + \infty \right[$. This is a proper surjective
  submersion
  by properties \ref{item:2} and \ref{item:3}; thus it is a locally
  trivial fiber bundle. Since its codomain is simply connected and its
  domain is connected, the long exact sequence in homotopy for the
  above restriction implies
  that, for all $c \in \left]
    \frac{\epsilon}{2}, + \infty \right[$, $H^{-1}_{\epsilon}(c)$ is
  connected, as desired.
\end{proof}

Proposition \ref{prop:basic_is} and Corollary
\ref{cor:compact_connected} provide a complete topological description
of the map $H_{\epsilon}$ for any $\epsilon > 0$: for any $c \in \left]
  \frac{\epsilon}{2}, + \infty \right[$, the fiber
$H^{-1}_{\epsilon}(c)$ is regular and diffeomorphic to $S^1$, while
$H^{-1}_{\epsilon}\left(\frac{\epsilon}{2}\right)$ is a point. \\

Next we study the symplectic geometry of
$H_{\epsilon}: B^1_{\infty} \times_L \R \to \R$ for a fixed $\epsilon >
0$. Since the regular fibers of $H_{\epsilon}$ are compact
and connected, the Liouville-Arnol'd theorem
ensures the existence of {\em local action-angle variables}
(cf. \cite[Section 50]{arnold}, \cite{duistermaat}, \cite[Section
44]{gui_ster_st} for details in general). For the
case at hand, this can be phrased as follows. By Properties \ref{item:1} and
\ref{item:2} in Proposition \ref{prop:basic_is}, the intersection of
the set of regular
values of $H_{\epsilon}$ with $H_{\epsilon}\left(B^1_{\infty} \times \R\right)$ equals
$\left]\frac{\epsilon}{2},+\infty \right[$. Then the Liouville-Arnol'd
theorem applied to $H_{\epsilon} : B^1_{\infty} \times \R \to \R$
yields the following result, stated below without proof.

\begin{lemma}\label{lemma:loc_aa}
  For any $\epsilon > 0$ and for any $c_0 \in
  \left]\frac{\epsilon}{2},+\infty \right[$, there exist an open
  neighborhood $U \subset \left]\frac{\epsilon}{2},+\infty \right[$ of
  $c_0$, a local diffeomorphism $I^U_{\epsilon} : U \to \R$, and a
  symplectomorphism $ \Psi^U_{\epsilon} : \left(H^{-1}_{\epsilon}(U),d x \wedge dy\right) \to
  \left(I^U_{\epsilon}(U) \times S^1, d a \wedge d\theta\right)$ such
  that $\left(\Psi^U_{\epsilon}, I^U_{\epsilon}\right)$ is an isomorphism
  between the subsystems of $H_{\epsilon} : B^1_\infty \times_L \R \to
  \R$ and $\mathrm{pr}_1 : \left(\R \times S^1, da \wedge
    d\theta\right) \to \R$ relative to
  $H^{-1}_{\epsilon}\left(U\right)$ and $\mathrm{pr}^{-1}\left(I_{\epsilon}\left(U\right)\right)$ respectively.
\end{lemma}

\begin{rmk}\label{rmk:loc_act}
  The smooth map $I^U_{\epsilon}$ of Lemma \ref{lemma:loc_aa} is referred to as a
  {\em local action} near $c_0$; an explicit, well-known formula for
  $I^U_{\epsilon}$ is given by
  \begin{equation}
    \label{eq:2}
    I^U_{\epsilon}(c)= \oint_{H_{\epsilon}^{-1}(c)}
    y(x,c) d x,
  \end{equation}
  \noindent
  where $y(x,c)$ is the smooth function defined implicitly by the
  equation $H_{\epsilon}(x,y) = c$ (cf. \cite[Section
  50]{arnold})\footnote{Equation \eqref{eq:2} differs by the standard formula for local
  actions by a factor of $2\pi$ (cf. \cite[Section 50]{arnold}). This
  is due to the fact that, in
  this paper, we identify $S^1$ with $\R/\Z$ while it is customary in
  the literature to use the identification $S^1 \cong \R/2\pi\Z$.}. Moreover, the map $I^U_{\epsilon} \circ H_{\epsilon} :
\left(H_{\epsilon}^{-1}(U), d x \wedge d y\right) \to \R$ is the moment
map of an effective Hamiltonian $S^1$-action. This is because $\left(\psi^U_{\epsilon}, I^U_{\epsilon}\right)$
is an isomorphism of integrable systems and $\left(I^U_{\epsilon}\left(U\right)
  \times S^1, da \wedge d\theta, \mathrm{pr}_1\right)$ is a
symplectic toric manifold (see Example \ref{exm:is}\ref{item:20}).
\end{rmk}

{\it A priori}, Lemma \ref{lemma:loc_aa} only holds locally, {\it
  i.e.} in a neighborhood of any given regular value. In general,
there are well-known topological obstructions to gluing these local
isomorphisms (cf. \cite{duistermaat}). However, in the
case at hand the situation is particularly simple.

\begin{cor}\label{cor:global}
  For any $\epsilon > 0$, there exist a smooth map $I_{\epsilon} :
  \left] \frac{\epsilon}{2}, + \infty \right[ \to \R$ which is a
  diffeomorphism onto its image, and a symplectomorphism
  $\Psi_{\epsilon} : \left(B^1_{\infty} \times_L \R\right) \smallsetminus \left\{(0,0)\right\}
  \to \left(I_{\epsilon}\left(\left] \frac{\epsilon}{2}, + \infty
      \right[ \right) \times S^1, da \wedge d\theta\right)$ such that
  $\left(\Psi_{\epsilon}, I_{\epsilon}\right)$ is an isomorphism
  between the subsystems of $H_{\epsilon} : B^1_\infty \times_L \R \to
  \R$ and $\mathrm{pr}_1 : \left(\R \times S^1, da \wedge
    d\theta\right) \to \R$ relative to
  $H^{-1}_{\epsilon}\left(\left] \frac{\epsilon}{2}, + \infty
    \right[\right)$ and $\mathrm{pr}_1^{-1}\left(I_{\epsilon}\left(\left] \frac{\epsilon}{2},
      + \infty \right[\right)\right)$ respectively.
  In particular, the map $I_{\epsilon} \circ
  H_{\epsilon}: \left(B^1_{\infty} \times_L \R\right) \smallsetminus \left\{(0,0)\right\} \to
  \R$ is the moment map of an effective Hamiltonian $S^1$-action.
\end{cor}

\begin{proof}
  Fix $\epsilon >0$. The topological obstructions to gluing the local
  isomorphisms of Lemma \ref{lemma:loc_aa} depend on the topology of
  the intersection of the set of regular values of $H_{\epsilon}$ with
  the image of $H_{\epsilon}$ (cf. \cite{duistermaat}). In particular,
  they vanish if this intersection is contractible. Therefore, since
  the
  intersection under consideration equals
  $\left]\frac{\epsilon}{2},+\infty \right[$, the local isomorphisms
  of Lemma \ref{lemma:loc_aa} can be glued together to obtain an
  isomorphism $\left(\Psi_\epsilon,I_\epsilon\right)$ between the subsystems of $H_{\epsilon} : B^1_\infty \times_L \R \to
  \R$ and $\mathrm{pr}_1 : \left(\R \times S^1, da \wedge
    d\theta\right) \to \R$ relative to
  $H^{-1}_{\epsilon}\left(\left] \frac{\epsilon}{2}, + \infty
    \right[\right)$ and $I_{\epsilon}\left(\left] \frac{\epsilon}{2},
      + \infty \right[\right)$ respectively. It remains to show that
  $I_{\epsilon}$ is a diffeomorphism onto its image. Since, for any
  open subset $U$ as in Lemma \ref{lemma:loc_aa}, $I_{\epsilon}|_U =
  I^U_{\epsilon}$, for any $c
  \in \left]\frac{\epsilon}{2},+\infty
  \right[$, $I'_{\epsilon}(c) \neq 0 $. Connectedness of $\left]
    \frac{\epsilon}{2}, + \infty \right[$ gives that $I_{\epsilon}$ is strictly monotone and, therefore, a diffeomorphism
  onto its image.
\end{proof}

\begin{rmk}\label{rmk:action}
  As a consequence of the proof of Corollary \ref{cor:global}, the right
  hand side of \eqref{eq:2} equals the function $I_{\epsilon}(c)$ for
  any $c \in \left] \frac{\epsilon}{2},+\infty
  \right[$. Substituting the function $y(x,c)$
  obtained by solving explicitly $H_{\epsilon}(x,y) = c$ in equation \eqref{eq:2},
  we obtain
  \begin{equation}
    \label{eq:1}
    I_{\epsilon}(c)  = 4 \int\limits_{0}^{\sqrt{1 - \frac{\epsilon}{2c}}} \sqrt{2c -
      \frac{\epsilon}{1-x^2}} d x = 4\sqrt{2c-\epsilon} \int\limits_{0}^{\sqrt{1 - \frac{\epsilon}{2c}}} \sqrt{1 -
      \frac{\epsilon x^2}{(2c-\epsilon)(1-x^2)}} d x.
  \end{equation}
  Formula \eqref{eq:1} gives that the action $I_{\epsilon}$ varies
  continuously with $\epsilon$.
\end{rmk}

For each $\epsilon > 0$, Corollary \ref{cor:global} describes the symplectic geometry of the
restriction of the integrable system $H_{\epsilon}:B^1_{\infty}
\times_L \R \to \R$ to its regular points. In fact, it is possible
to strengthen Corollary \ref{cor:global} to provide a description of
the integrable system that includes its singular point; this can be
achieved by exploiting the {\em
  linearization} results for non-degenerate singular elliptic points
(cf. \cite{duf_mol,eliasson} for further details in general). For the
purposes at hand, it suffices to state the linearization result in the
simplest case, which is a consequence of the main theorem in
\cite{cdvv}.

\begin{thm}[Colin de Verdi\`ere and Vey, \cite{cdvv}]\label{thm:elliptic}
  Let $H : \left(\R^2, dx \wedge dy\right) \to \R$ be an integrable system such that $(0,0)$ is a singular point of $H$ and
  the Hessian of $H$ at $(0,0)$ is positive definite. Then there exist
  open neighborhoods $U \subset H(\R^2)$, $V \subset
  \left[0,+\infty\right[$ of $H(0,0)$ and of $0$ respectively, a local
  diffeomorphism $I : U \to V$ with $I(H(0,0)) = 0$, and a
  symplectomorphism $\Psi : \left(H^{-1}(U), d x \wedge d y\right) \to
  \left(\mu^{-1}(V),d u \wedge dv \right)$ such that
  $\left(\Psi,I\right)$ is an isomorphism between the subsystems of
  $H : \left(\R^2, dx \wedge dy\right) \to \R$ and of $\mu :
  \left(\R^2, du \wedge dv\right) \to \R$ relative to $H^{-1}(U)$ and
  $\mu^{-1}(V)$ respectively, where $\mu(u,v) = \pi\left(u^2 + v^2\right)$. In particular, $I
  \circ H : \left(H^{-1}(U), d x \wedge d y\right) \to \R$ is the
  moment map of an effective Hamiltonian $S^1$-action.
\end{thm}


Applying Theorem \ref{thm:elliptic} to the family of integrable
systems $H_{\epsilon}: B^1_{\infty} \times_L \R \to \R$, we obtain the
following result.

\begin{lemma}\label{lemma:extend}
  For each $\epsilon > 0$, the action function $I_{\epsilon}: \left]
    \frac{\epsilon}{2}, + \infty \right[ \to \R$ of Corollary
  \ref{cor:global} extends to a smooth function defined on $\left[\frac{\epsilon}{2}, + \infty \right[$.
\end{lemma}
\begin{proof}
  Fix $\epsilon > 0$. The integrable system $H_{\epsilon}
  :B^1_{\infty} \times_L \R \to \R$ satisfies the hypotheses of Theorem
  \ref{thm:elliptic}. Therefore it is possible to find open
  neighborhoods $U \subset  \left[\frac{\epsilon}{2}, + \infty
  \right[$ and $V \subset \left[0,+\infty\right[$ of $\frac{\epsilon}{2}$ and $0$
  respectively, and an isomorphism $\left(\Psi,I\right)$ between the
  subsystems of $H_{\epsilon} : B^1_{\infty} \times_L \R \to \R$ and of $\mu :
  \left(\R^2, du \wedge dv\right) \to \R$ relative to $H_{\epsilon}^{-1}(U)$ and
  $\mu^{-1}(V)$ respectively. Shrinking $U$ if needed, it may be
  assumed that $U$ is connected and that $I$ is a diffeomorphism onto
  $V$. By abuse of notation denote the restrictions of $I$ and
  $\Psi$ to $U \cap \left]\frac{\epsilon}{2}, + \infty \right[$ and
  $H^{-1}_{\epsilon}\left(U \cap \left]\frac{\epsilon}{2}, + \infty
    \right[\right)$ respectively by $I$ and $\Psi$. 

  Since $\mathrm{pr}_1 \circ \Psi_{\epsilon} \circ \Psi^{-1}$ is the
  moment map of an effective Hamiltonian $S^1$-action, so is
  $I_{\epsilon} \circ I^{-1} \circ \mu$. Moreover, if
  $\mathcal{X}_{\mu}$ and $\mathcal{X}_{I_{\epsilon} \circ I^{-1}
    \circ \mu}$ denote the Hamiltonian vector fields of the functions
  $\mu$ and $\mathcal{X}_{I_{\epsilon} \circ I^{-1}
    \circ \mu}$ respectively, then $\mathcal{X}_{I_{\epsilon} \circ I^{-1}
    \circ \mu} = \left(\frac{d \left(I_{\epsilon} \circ I^{-1}\right)}{d c}
    \circ \mu\right) \mathcal{X}_{\mu}$. Since $\mu$ and $\frac{d \left(I_{\epsilon} \circ I^{-1}\right)}{d c}
  \circ \mu$ Poisson commute and are moment maps of effective
  Hamiltonian $S^1$-actions, it follows that the
  function $\frac{d \left(I_{\epsilon} \circ I^{-1}\right)}{d c}
  \circ \mu$ takes values in $\left\{\pm 1\right\}$. Since
  $\mu^{-1}(V) \smallsetminus \left\{(0,0)\right\}$ is connected, then $\frac{d \left(I_{\epsilon} \circ I^{-1}\right)}{d c}
  \circ \mu$ is constant. Moreover, since $\mu$ is a submersion
  restricted to $\mu^{-1}(V) \smallsetminus \left\{(0,0)\right\}$, it
  follows that $\frac{d \left(I_{\epsilon} \circ I^{-1}\right)}{d c}$
  is constant and equal to $\pm 1$. Thus the function $I_{\epsilon}
  \circ I^{-1}$ is the restriction of an element $h$ of $\mathrm{AGL}(1;\Z):=
  \mathrm{GL}(1;\Z) \ltimes \R$ to $V \smallsetminus \{0\}$. In
  particular, since $I$ can be extended smoothly at $\frac{\epsilon}{2}$, so can
  $I_{\epsilon}$, which proves the desired result.
\end{proof}

By abuse of notation, denote the extension given by Lemma
\ref{lemma:extend} also by $I_{\epsilon} :
\left[\frac{\epsilon}{2}, +\infty \right[ \to \R$.

\begin{lemma}\label{lemma:action}
  For a fixed $\epsilon > 0$, the map $I_{\epsilon}:
  \left[\frac{\epsilon}{2}, + \infty\right[ \to [0,+\infty[$ is a
  diffeomorphism.
\end{lemma}
\begin{proof}
  Fix $\epsilon > 0$. The proofs of Theorem \ref{thm:elliptic} and of Lemma
\ref{lemma:extend} imply that the derivative of $I_{\epsilon}$ at
$\frac{\epsilon}{2}$ does not vanish, which, together with Corollary
\ref{cor:global}, gives that $I'_{\epsilon}$ does not vanish on $\left[\frac{\epsilon}{2}, +\infty \right[$. To prove the desired result, it suffices to
  show that $I_{\epsilon} \left(\frac{\epsilon}{2}\right) = 0$, that
  $I_{\epsilon}$ is strictly increasing, and that the image of $I_{\epsilon}$ is
  not bounded. To prove the first result, we have to show that $ \lim\limits_{c \to
    \frac{\epsilon}{2}^+} I_{\epsilon}(c) = 0$. Using equation
  \eqref{eq:1}, it suffices to prove that the integral $\int\limits_{0}^{\sqrt{1 - \frac{\epsilon}{2c}}} \sqrt{1 -
    \frac{\epsilon x^2}{(2c-\epsilon)(1-x^2)}} d x$ is bounded. Since the integrand
  is non-negative, this integral is certainly non-negative; on the
  other hand, the integrand is less than 1, which implies that the
  integral is, in fact, bounded as required. This shows that
  $I_{\epsilon} \left(\frac{\epsilon}{2}\right) = 0$.

  Let $c_1 > c_2 \geq \frac{\epsilon}{2}$. Then
  \begin{equation*}
    \begin{split}
      I_{\epsilon}(c_1) &= 4 \int\limits_{0}^{\sqrt{1 - \frac{\epsilon}{2c_1}}} \sqrt{2c_1 -
        \frac{\epsilon}{1-x^2}} d x  > 4 \int\limits_{0}^{\sqrt{1 - \frac{\epsilon}{2c_2}}} \sqrt{2c_1 -
        \frac{\epsilon}{1-x^2}} d x \\
      &> 4 \int\limits_{0}^{\sqrt{1 - \frac{\epsilon}{2c_2}}} \sqrt{2c_2 -
        \frac{\epsilon}{1-x^2}} d x = I_{\epsilon}(c_2),
    \end{split}
  \end{equation*}
  \noindent
  where the first inequality follows from the fact that the function $\sqrt{2c_1 -
    \frac{\epsilon}{1-x^2}}$ is positive on $\left[\sqrt{1 -
      \frac{\epsilon}{2c_2}}, \sqrt{1 -
      \frac{\epsilon}{2c_1}}\right]$, while the second is a
  consequence of the fact
  that $c_1 > c_2$ implies that, for all  $ x \in \left[0, \sqrt{1 -
      \frac{\epsilon}{2c_2}}\right]$, $\sqrt{2c_1 -
    \frac{\epsilon}{1-x^2}} > \sqrt{2c_2 -
    \frac{\epsilon}{1-x^2}}$. Therefore, $I_{\epsilon}$ is strictly
  increasing as desired.

  Finally, to see that $I_{\epsilon}$ is unbounded, observe that, by
  equation \eqref{eq:1}, it suffices to show that, for all $c $
  sufficiently large, the integral
  $\int\limits_{0}^{\sqrt{1 - \frac{\epsilon}{2c}}} \sqrt{1 -
    \frac{\epsilon x^2}{(2c-\epsilon)(1-x^2)}} d x$
  is bounded away from 0. To this end, observe that this integral
  depends continuously on $c$ so that, as $c \to + \infty$, the above
  integral tends to
  $1$, thus implying the desired property.
\end{proof}

In order to strengthen Corollary \ref{cor:global} to include the
singular point of $H_{\epsilon} : B^1_\infty \times_L \R \to \R$, we
need the following result, which is a consequence of \cite[Theorem
1.3]{kar_ler} (and a generalization of the well-known classification
of {\em compact} symplectic toric manifolds due to Delzant, cf. \cite{delzant}), and is stated below without proof.

\begin{thm}\label{thm:class_stm}
  For $i=1,2$, let $\left(M_i,\omega_i,\mu_i\right)$ be a symplectic
  toric manifold with connected fibers with $\mu_i\left(M_i\right)$
  contractible. Then there exists a symplectomorphism $\Psi :
  \left(M_1,\omega_1\right) \to \left(M_2,\omega_2\right)$ with $\mu_2
  \circ \Psi = \mu_1$ if and only if $\mu_1\left(M_1\right) =
  \mu_2\left(M_2\right)$.
\end{thm}

\begin{rmk}\label{rmk:cont_fam}
  While not explicitly stated in \cite{kar_ler}, it follows from ideas
  therein that if $\left\{\left(M,\omega,
      \mu_{\epsilon}\right)\right\}_{\epsilon >0}$ is a family of
  symplectic toric manifolds depending continuously on a parameter
  $\epsilon$ such that
  \begin{itemize}[leftmargin=*]
  \item for all $\epsilon > 0$, the fibers of $\mu_\epsilon$ are
    connected, and
  \item there exists a symplectic toric manifold
    $\left(M',\omega',\mu'\right)$ with $\mu'\left(M'\right) =
    \mu_\epsilon\left(M\right)$ for all $\epsilon >0$,
  \end{itemize}
  then the family of symplectomorphisms $\Psi_\epsilon :
  \left(M,\omega\right) \to \left(M',\omega'\right)$ can be chosen to
  depend continuously on $\epsilon$.
\end{rmk}

In analogy with Corollary \ref{cor:global}, we have the following
result describing the symplectic geometry of the integrable
system $H_{\epsilon} :B^1_{\infty} \times_L \R\to \R$.

\begin{cor}\label{cor:glob_sing}
  For any $\epsilon > 0$, there exists a symplectomorphism
  $\Psi_{\epsilon} : B^1_{\infty} \times_L \R \to \left(\R^2 , du
    \wedge d v\right)$ such that $\left(\Psi_{\epsilon},
    I_{\epsilon}\right)$ is an isomorphism between $H_{\epsilon} :
  B^1_{\infty} \times_L \R\to \R$ and $\mu: \left(\R^2 , du
    \wedge d v\right) \to \R$, where $\mu(u,v) =
  \pi \left(u^2 + v^2\right)$. Moreover, the family
  $\left\{\Psi_{\epsilon}\right\}_{\epsilon > 0}$ may be chosen to depend continuously
  on $\epsilon$.
\end{cor}

\begin{proof}
  Fix $\epsilon > 0$. By construction, the composite $I_{\epsilon}
  \circ H_{\epsilon}$ is the moment map of an effective Hamiltonian
  $S^1$-action with connected fibers whose image equals $\left[0, +
    \infty\right[$ by Lemma \ref{lemma:action}. Thus
  $\left(B^1_{\infty} \times_L \R, dx \wedge dy, I_{\epsilon} \circ
    H_\epsilon\right)$ and $\left(\R^2, du
    \wedge dv,\mu\right)$ are symplectic toric manifolds satisfying
  the hypotheses of Theorem \ref{thm:class_stm}. Seeing as they have equal moment map
  images, Theorem \ref{thm:class_stm} ensures the existence of the desired symplectomorphism
  $\Psi_{\epsilon}$. The fact that $\Psi_{\epsilon}$ may be chosen to
  depend continuously on $\epsilon$ follows from Remark \ref{rmk:cont_fam}.
\end{proof}

An important consequence of Corollary \ref{cor:glob_sing}, which plays
a key role in the proof of Theorem \ref{thm:sympl}, is the following
result.

\begin{prop}\label{prop:decreasing_epsilon}
  For all $\epsilon_1 > \epsilon_2$ and for all $c \geq
  \frac{\epsilon_1}{2}$,
  \begin{equation}
    \label{eq:8}
    \Psi_{\epsilon_1}\left(H^{-1}_{\epsilon_1}\left(\left[
          \frac{\epsilon_1}{2},c \right[ \,\right)\right) \subset  \Psi_{\epsilon_2}\left(H^{-1}_{\epsilon_2}\left(\left[
          \frac{\epsilon_2}{2},c \right[ \,\right)\right).
  \end{equation}
\end{prop}
\begin{proof}
  Fix $\epsilon_1 > \epsilon_2$ and $c \geq
  \frac{\epsilon_1}{2}$. Firstly, observe
  that
  \eqref{eq:8} is equivalent to
  \begin{equation}
    \label{eq:10}
    \mu^{-1}\left(\left[0, I_{\epsilon_1}(c) \right[ \right) \subset
    \mu^{-1}\left( \left[0, I_{\epsilon_2}(c) \right[\right).
  \end{equation}
  \noindent
  This can be seen as follows: for
  $i=1,2$, we have that
  \begin{equation}
    \label{eq:9}
    \mu \circ \Psi_{\epsilon_i}\left(H^{-1}_{\epsilon_i}\left(\left[
          \frac{\epsilon_i}{2},c \right[ \,\right)\right) = I_{\epsilon_i}\left(\left[
        \frac{\epsilon_i}{2},c \right[ \,\right) =
    \left[0,I_{\epsilon_i}(c)\right[;
  \end{equation}
  \noindent
  the first equality follows from Corollary \ref{cor:glob_sing}, while
  the second from Lemma \ref{lemma:action}. Corollary
  \ref{cor:glob_sing} also implies that, for $i=1,2$, the subset $\Psi_{\epsilon_i}\left(H^{-1}_{\epsilon_i}\left(\left[
        \frac{\epsilon_i}{2},c \right[ \,\right)\right)$ is saturated
  with respect to $\mu$. This fact, together with equation
  \eqref{eq:9} implies the inclusion of equation \eqref{eq:8} holds if
  and only if that of equation \eqref{eq:10} does. To show that
  equation \eqref{eq:10} is true, it suffices to prove that
  $\left[0,
    I_{\epsilon_1}(c) \right[  \subset \left[0, I_{\epsilon_2}(c)
  \right[ $ or, equivalently, that $I_{\epsilon_1}(c) <
  I_{\epsilon_2}(c)$.  The proof of this last statement is analogous
  to an argument used in Lemma
  \ref{lemma:action}. Using equation \eqref{eq:1}, it can be seen
  that, for fixed $c$,
  $I_{\epsilon}(c)$ is a continuous, decreasing function of
  $\epsilon$. This yields the desired result.
\end{proof}

To conclude this section, we observe that equation \eqref{eq:1} implies that, in some sense, the
family of diffeomorphisms $\left\{I_{\epsilon}\right\}_{\epsilon > 0}$
converges uniformly as $\epsilon$ goes to 0.

\begin{lemma}\label{lemma:limit}
  For all $c >0$, $\lim\limits_{\epsilon \to
    0^+} I_{\epsilon}(c) = 4\sqrt{2c} =: I_0(c)$. Moreover, for any $\epsilon_0 >0$, any decreasing sequence $\epsilon_k$
  converging to 0 with the property that $\epsilon_0 > \epsilon_1$,
  and any compact subset $K \subset \R_{\geq 0}$, $I_{\epsilon_k} \to
  I_0$ uniformly in the set $K \cap \left[\frac{\epsilon_0}{2},+\infty\right[$.
\end{lemma}
\begin{proof}
  Fix $c >0$. Then $c$ is in the domain of $I_{\epsilon}$ for all
  $\epsilon$ sufficiently small; therefore, it makes sense to consider $\lim\limits_{\epsilon \to
    0^+} I_{\epsilon}(c)$. The result follows from observing
  that $I_{\epsilon}$ depends continuously on $\epsilon$; thus equation
  \eqref{eq:1} yields that
  $$ \lim\limits_{\epsilon \to
    0^+} I_{\epsilon}(c) = 4 \int\limits_0^1 \sqrt{2c} \, d x
  = 4 \sqrt{2c}. $$
  \noindent
  This proves the first assertion. To prove the second, fix $\epsilon_0$, a decreasing sequence $\epsilon_k$ converging
  to $0$ and a compact set $K$ as in the statement. Then $K':= K \cap
  \left[\frac{\epsilon_0}{2},+\infty\right[$ is compact, the family of functions
  $\left\{I_{\epsilon_k}|_{K'}\right\}_n$ is monotone (see the proof
  of Proposition \ref{prop:decreasing_epsilon}), and the function
  $I_0|_{K'}$ is continuous. The result then follows by Dini's theorem.
\end{proof}


\subsubsection{The general case}\label{sec:general-case}
For any $n \geq 1$, consider the family of smooth maps
$\left\{ \Phi_{\epsilon} : B^n_{\infty} \times_L \R^n \to \R^n \right\}_{\epsilon > 0}$, where
\begin{equation}
  \label{eq:12}
  \Phi_{\epsilon}\left(\mathbf{x},\mathbf{y}\right) =
  \left(H_{\epsilon}(x_1,y_1), \ldots, H_{\epsilon}(x_n,y_n)\right),
\end{equation}
\noindent
and $H_{\epsilon}: B^1_{\infty} \times_L \R \to \R$ is the smooth function introduced in Section
\ref{sec:billiard-interval}. Viewing $B^n_{\infty} \times_L \R^n$ as
the symplectic product of $n$ copies of $B^1_{\infty} \times_L \R$, it
follows from the construction \ref{item:19} in Example \ref{exm:is}
that, for each $\epsilon >0$, $\Phi_{\epsilon} : B^n_{\infty} \times_L
\R^n \to \R^n $ is an integrable system. In fact, much more is true.


\begin{cor}\label{cor:glob_sing_general}
  For any $\epsilon > 0$, there exist a diffeomorphism
  $\mathbf{I}_{\epsilon} : \left[\frac{\epsilon}{2},+\infty
  \right[^n \to \left[0,+\infty\right[^n$ and a symplectomorphism
  $\boldsymbol{\Psi}_{\epsilon} : B^n_{\infty} \times_L \R^n \to
  \left(\R^{2n} , \omega_0\right)$ such that
  $\left(\boldsymbol{\Psi}_{\epsilon}, \mathbf{I}_{\epsilon}\right)$
  is an isomorphism between $\Phi_{\epsilon} : B^n_{\infty} \times_L
  \R^n \to \R^n$ and $\boldsymbol{\mu} :
  \left(\R^{2n},\sum\limits_{i=1}^n d u_i \wedge dv_i\right) \to
  \R^n$, where $\boldsymbol{\mu}(\mathbf{u},\mathbf{v}) =
  \pi \left(u_1^2 + v_1^2,\ldots,u^2_n + v^2_n\right)$. In particular, $\mathbf{I}_{\epsilon}
  \circ \Phi_{\epsilon}$ is the moment map of an effective Hamiltonian
  $\mathbb{T}^n$-action on $B^n \times_L \R^n$. Moreover, the family $\left\{\boldsymbol{\Psi}_{\epsilon}\right\}_{\epsilon > 0}$ may be chosen to depend continuously
  on $\epsilon$.
\end{cor}
\begin{proof}
  Setting $\mathbf{I}_{\epsilon}\left(\mathbf{c}\right):=
  \left(I_{\epsilon}(c_1),\ldots,I_{\epsilon}(c_n)\right)$ and
  $\boldsymbol{\Psi}_{\epsilon}\left(\mathbf{x},\mathbf{y}\right)
  :=\left(\Psi_{\epsilon}(x_1,y_1),\ldots,
    \Psi_{\epsilon}(x_n,y_n)\right)$, where
  $I_{\epsilon}$ and $\Psi_{\epsilon}$ are as in equation \eqref{eq:1}
  and Corollary
  \ref{cor:glob_sing}
  respectively, the desired result follows from Remark
  \ref{rmk:im_prod} and Corollary \ref{cor:glob_sing}.
\end{proof}


Lemma \ref{lemma:limit} and Corollary \ref{cor:glob_sing_general}
imply that the family of diffeomorphisms
$\left\{\mathbf{I}_{\epsilon}\right\}_{\epsilon > 0}$ converges
uniformly as $\epsilon$ goes to 0.

\begin{cor}\label{cor:limit_general_case}
  For any $\mathbf{c} \in \left]0,+\infty\right[^n$, $\lim\limits_{\epsilon
    \to 0^+} \mathbf{I}_{\epsilon}\left(\mathbf{c}\right) =
  \mathbf{I}_0\left(\mathbf{c}\right)$, where
  $\mathbf{I}_0\left(\mathbf{c}\right):=\left(I_0(c_1),\ldots,I_0(c_n)\right)$
  and $I_0(c) = 4\sqrt{2c}$. Moreover, for any $\epsilon_0 > 0$, any decreasing sequence $\epsilon_k$
  converging to 0 with $\epsilon_0 > \epsilon_1$, and any compact
  subset $K \subset \R^n_{\geq 0}$, $\mathbf{I}_{\epsilon_k} \to
  \mathbf{I}_0$ uniformly in $K \cap \left[\frac{\epsilon_0}{2},+\infty\right[^n$.
\end{cor}

\begin{proof}
  The first statement is an immediate consequence of Lemma
  \ref{lemma:limit} and Corollary \ref{cor:glob_sing_general}. The
  second statement follows similarly upon
  observing that, without loss of generality, it may be assumed that
  $K$ is of the form $K_1 \times \ldots \times K_n \subset \R_{\geq 0}
  \times \ldots \times \R_{\geq 0} = \R^n_{\geq 0}$, where, for each
  $i = 1,\ldots,n$, $K_i \subset  \R_{\geq 0}$ is compact.
\end{proof}

\subsection{Constructing the symplectomorphism}\label{sec:constr-sympl}
The aim of this section is to prove Theorem \ref{thm:sympl}, which endows any lagrangian product of the form
$B^n_\infty \times_L A$, where $A \subset \R^n$ is a balanced region (see
Definition \ref{def:sym}), with an effective
Hamiltonian $\mathbb{T}^n$-action. As a first step, we construct a suitable
compact exhaustion of any lagrangian product of the above form (see
Step \ref{item:16}). Henceforth, given $B \subset \R^l$, we denote
its closure by $\mathrm{cl}\left(B\right)$.

\begin{lemma}\label{lemma:family}
  For any balanced region $A \subset \R^n$, there exists a family of
  symplectic submanifolds
  $\left\{P_{\epsilon}\right\}_{\epsilon >0}$ of $B^n_{\infty}
  \times_L A$, with compact closure  in $B^n_{\infty} \times_L A$, satisfying the following properties:
  \begin{enumerate}[label = (\alph*), ref = (\alph*), leftmargin=*]
  \item \label{item:7} $\bigcup\limits_{\epsilon >0} \mathrm{cl}\left(P_{\epsilon}\right) = B^n_{\infty}
    \times_L A$ and $\bigcup\limits_{\epsilon >0}
    \boldsymbol{\Psi}_{\epsilon}\left(\mathrm{cl}\left(P_{\epsilon}\right)\right) =
    X_{4|A|}$;
  \item \label{item:6} if $\epsilon_1 > \epsilon_2$, then
    $\mathrm{cl}\left(P_{\epsilon_1}\right) \subset \mathrm{cl}\left(P_{\epsilon_2}\right)$ and
    $\boldsymbol{\Psi}_{\epsilon_1}\left(\mathrm{cl}\left(P_{\epsilon_1}\right)\right)
    \subset \boldsymbol{\Psi}_{\epsilon_2}\left(\mathrm{cl}\left(P_{\epsilon_2}\right)\right)$,
  \end{enumerate}
  where $4|A| \subset \R_{\geq 0}^n$ is as in Section \ref{sec:tor}, and $\left\{\boldsymbol{\Psi}_{\epsilon} :
  B^n_{\infty} \times_L \Sigma \to \R^{2n}\right\}_{\epsilon >0}$ is
  the family of symplectomorphisms of Corollary
  \ref{cor:glob_sing_general} depending continuously on $\epsilon$.
\end{lemma}

\begin{proof}
  Fix a balanced region $A \subset \R^n$. For any $\epsilon > 0$, let $\Phi_{\epsilon}
  : B^n_{\infty} \times_L \R^n \to \R^n$ be the integrable system
  defined by equation \eqref{eq:12}. For $\epsilon > 0$, set
  $$ P_{\epsilon}:=
  \Phi_{\epsilon}^{-1}\left(\mathbf{I}^{-1}_0\left(4|A|\right)\right)
  \subset B^n_{\infty} \times_L \R^n, $$
  \noindent
  where $\mathbf{I}_0 : \R^n_{\geq 0} \to \R^n_{\geq 0}$ is the
  map of Corollary \ref{cor:limit_general_case}.
  The claim is that $\left\{P_{\epsilon}\right\}_{\epsilon > 0}$ is
  the required family. Begin by observing that, since $A$ is open, so
  is $4|A| \subset \R^n_{\geq 0}$. Continuity of $\Phi_{\epsilon}$ for
  any $\epsilon >0$ and of $\mathbf{I}_0$ implies that, for each
  $\epsilon >0$, $P_{\epsilon}$ is an
  open subset of $B^n_{\infty} \times_L \R^n$ and, thus, a symplectic
  submanifold of $B^n_{\infty} \times_L \R^n$. For each $\epsilon >
  0$, the closure of $P_{\epsilon}$ is mapped to the closure of $4|A|$
  in $\R^n_{\geq 0}$ under $\mathbf{I}_0 \circ \Phi_{\epsilon}$. Since $A$ is bounded,
  so is $4|A|$ is bounded, which implies that $\mathrm{cl}\left(4|A|\right)
  \subset \R^n_{\geq 0}$ is compact. Moreover, the maps
  $\Phi_{\epsilon}$ and $\mathbf{I}_0$ are proper, the former by
  Property \ref{item:3} of Proposition \ref{prop:basic_is} and by
  construction, while the latter by virtue of being a homeomorphism
  onto a closed subset of $\R^n$. Therefore, for each $\epsilon >0$,
  the closure of $P_{\epsilon}$ is contained in a compact subset and
  is, therefore, compact. To simplify the
  argument of the rest of the proof, we deal with each statement separately.

  \begin{claim}\label{claim:union_P}
    $\bigcup\limits_{\epsilon >0} \mathrm{cl}\left(P_{\epsilon}\right) = B^n_{\infty}
    \times_L A$.
  \end{claim}
  \begin{proof}[Proof of Claim \ref{claim:union_P}]
    Fix $\epsilon >0$ and let
  $\left(\mathbf{x},\mathbf{y}\right) \in \mathrm{cl}\left(P_{\epsilon}\right)$. By
  definition, $\mathbf{x} \in B^n_{\infty}$ and
  $\mathbf{I}_0
  \left(\Phi_{\epsilon}\left(\mathbf{x},\mathbf{y}\right)\right) \in
  \mathrm{cl}\left(4|A|\right)$. Using the definition of $\mathbf{I}_0$ and
  $\Phi_{\epsilon}$, the latter condition gives that
  \begin{equation}
    \label{eq:13}
    \left(4\sqrt{y^2_1 + \frac{\epsilon}{1-x^2_1}}, \ldots,
      4\sqrt{y^2_n + \frac{\epsilon}{1-x^2_n}}\right) \in \mathrm{cl}\left(4|A|\right).
  \end{equation}
  \noindent
  However, since $4|A|$ satisfies \eqref{eq:cond}, equation
  \eqref{eq:13} implies that
  $$ \left[0,4\sqrt{y^2_1 + \frac{\epsilon}{1-x^2_1}}\right[ \times
  \ldots \times  \left[0,4\sqrt{y^2_n +
      \frac{\epsilon}{1-x^2_n}}\right[ \subset 4|A|,$$
  \noindent
  which, in particular, yields that $\left(4\lvert y_1\rvert, \ldots,
    4\lvert y_n\rvert\right) \in 4|A|$. By definition of
  $4|A|$, this last condition gives that $\mathbf{y} \in
  A$. Thus $\left(\mathbf{x},\mathbf{y}\right) \in
  B^n_{\infty} \times_L A$; since
  $\left(\mathbf{x},\mathbf{y}\right) \in \mathrm{cl}\left(P_{\epsilon}\right)$ and $\epsilon
  > 0$ are arbitrary,
  for all $\epsilon > 0$, $\mathrm{cl}\left(P_{\epsilon}\right) \subset B^n_{\infty} \times_L
  A$. Hence, for all $\epsilon > 0$, $P_{\epsilon}$ is a
  symplectic submanifold of $B^n_{\infty} \times_L A$ with compact
  closure in $B^n_\infty \times_L A$,
  and $\bigcup\limits_{\epsilon >0} \mathrm{cl}\left(P_{\epsilon}\right) \subset B^n_{\infty}
  \times_L A$.

  It remains to prove the opposite inclusion. Suppose that
  $\left(\mathbf{x},\mathbf{y}\right) \in B^n_{\infty} \times_L
  A$. Then, by definition, $\left(4\lvert y_1\rvert, \ldots,
    4\lvert y_n\rvert\right) \in 4|A|$; since $4|A| \subset
  \R^n_{\geq 0}$ is open, for all sufficiently small $\epsilon > 0$,
  $\left(4\sqrt{y^2_1 + \frac{\epsilon}{1-x^2_1}}, \ldots,
    4\sqrt{y^2_n + \frac{\epsilon}{1-x^2_n}}\right) \in 4|A|$, which
  is equivalent to $\left(\mathbf{x},\mathbf{y}\right) \in
  P_{\epsilon}$. Since $\left(\mathbf{x},\mathbf{y}\right) \in
  B^n_{\infty} \times_L A$ is arbitrary, this gives that $B^n_{\infty}
  \times_L A \subset \bigcup\limits_{\epsilon >0}
  P_{\epsilon} \subset \bigcup\limits_{\epsilon >0}
  \mathrm{cl}\left(P_{\epsilon}\right)$.
  \end{proof}

  \begin{claim}\label{claim:union_images}
    $\bigcup\limits_{\epsilon >0}
    \boldsymbol{\Psi}_{\epsilon}\left(\mathrm{cl}\left(P_{\epsilon}\right)\right) =
    X_{4|A|}$.
  \end{claim}
  \begin{proof}[Proof of Claim \ref{claim:union_images}]
    Fix $\epsilon > 0$; firstly we show that
  $\boldsymbol{\Psi}_{\epsilon}\left(\mathrm{cl}\left(P_{\epsilon}\right)\right) \subset
  X_{4|A|}$. Since $X_{4|A|}$ is saturated with respect to
  $\boldsymbol{\mu}$, it suffices to prove that $\boldsymbol{\mu}
  \left(\boldsymbol{\Psi}_{\epsilon}\left(\mathrm{cl}\left(P_{\epsilon}\right)\right)\right)
  \subset \boldsymbol{\mu}\left(X_{4|A|}\right) = 4|A|$. By Corollary
  \ref{cor:glob_sing_general}, $\boldsymbol{\mu} \circ
  \boldsymbol{\Psi}_{\epsilon} = \mathbf{I}_{\epsilon} \circ
  \Phi_{\epsilon}$ and, by definition, $P_{\epsilon} =
  \Phi^{-1}_{\epsilon}\left(\mathbf{I}^{-1}_0\left(4|A|\right)\right)$,
  so that $\Phi_{\epsilon}\left(\mathrm{cl}\left(P_{\epsilon}\right)\right)
  \subset \mathbf{I}^{-1}_0\left(\mathrm{cl}\left(4|A|\right)\right)$;
  therefore it suffices to prove that $
  \mathbf{I}_{\epsilon}\left(\mathbf{I}_0^{-1}\left(\mathrm{cl}\left(4|A|\right)\right)\right)
  \subset 4|A|$. In fact, since the domain of
  $\mathbf{I}_{\epsilon}$ is
  $\left[\frac{\epsilon}{2},+\infty\right[^n$, it suffices to show
  that
  $$ \mathbf{I}_{\epsilon}\left(\mathbf{I}_0^{-1}\left(\mathrm{cl}\left(4|A|\right)\right)
    \cap \left[\frac{\epsilon}{2},+\infty\right[^n\right)
  \subset 4|A|. $$
  \noindent
  Suppose that $\mathbf{a} \in \mathrm{cl}\left(4|A|\right)$ is such that
  $\mathbf{I}^{-1}_0\left(\mathbf{a}\right) \in
  \left[\frac{\epsilon}{2},+\infty\right[^n$; the aim is to show that
  $\mathbf{I}_{\epsilon}\left(\mathbf{I}^{-1}_0\left(\mathbf{a}\right)\right)
  \in 4|A|$. Observe that, by definition of $\mathbf{I}_0$ (see Corollary \ref{cor:limit_general_case}),
  $\mathbf{I}_0^{-1}\left(\mathbf{a}\right) =
  \left(I_0^{-1}\left(a_1\right),\ldots,I_0^{-1}\left(a_n\right)\right)$;
  moreover, by definition of $\mathbf{I}_{\epsilon}$ (see the proof of
  Corollary \ref{cor:glob_sing_general}),
  $$\mathbf{I}_{\epsilon}\left(\mathbf{I}^{-1}_0\left(\mathbf{a}\right)\right)
  = \left(I_{\epsilon}\left(I_0^{-1}\left(a_1\right)\right),\ldots,
    I_{\epsilon}\left(I_0^{-1}\left(a_n\right)\right)\right).$$
  \noindent
  By assumption, for each $i =1,\ldots, n$, $I^{-1}_0(a_i) \geq
  \frac{\epsilon}{2}$. The definitions of $I_{\epsilon}$ and of $I_0$
  (see \eqref{eq:1} and Lemma \ref{lemma:limit}) imply that,
  for all $i=1,\ldots,n$, $I_{\epsilon}\left(I^{-1}_0(a_i)\right) <
  I_0 \left(I^{-1}_0(a_i)\right) = a_i$. In particular,
  $$
  \mathbf{I}_{\epsilon}\left(\mathbf{I}^{-1}_0\left(\mathbf{a}\right)\right)
  \in \left[0,a_1\right[ \times \ldots \times \left[0,a_n\right[; $$
  \noindent
  on the other hand, the right hand side of the above equation is a subset of $4|A|$
  since $\mathbf{a} \in \mathrm{cl}\left(4|A|\right)$ and $4|A|$ satisfies
  \eqref{eq:cond}. Thus
  $\mathbf{I}_{\epsilon}\left(\mathbf{I}^{-1}_0\left(\mathbf{a}\right)\right)
  \in 4|A|$; since $\mathbf{a} \in 4|A|$ is arbitrary, the above
  argument shows that $\mathbf{I}_{\epsilon}\left(\mathbf{I}_0^{-1}\left(\mathrm{cl}\left(4|A|\right)\right)\right)
  \subset 4|A|$ and, therefore, $\boldsymbol{\Psi}_{\epsilon}\left(\mathrm{cl}\left(P_{\epsilon}\right)\right) \subset
  X_{4|A|}$. Since $\epsilon >0$ is arbitrary, we have that $\bigcup\limits_{\epsilon >0}
  \boldsymbol{\Psi}_{\epsilon}\left(\mathrm{cl}\left(P_{\epsilon}\right)\right) \subset
  X_{4|A|}$.

  To prove the opposite inclusion, suppose that
  $\mathbf{z} \in X_{4|A|}$; it suffices to show that there exists
  $\epsilon > 0$ and $\left(\mathbf{x},\mathbf{y}\right) \in
  P_{\epsilon} \subset \mathrm{cl}\left(P_{\epsilon}\right)$ such that
  $\boldsymbol{\Psi}_{\epsilon}\left(\mathbf{x},\mathbf{y}\right) =
  \mathbf{z}$. By Corollary \ref{cor:glob_sing_general}, we know that,
  for any $\epsilon > 0$, there exists a unique point
  $\left(\mathbf{x}_{\epsilon},\mathbf{y}_{\epsilon}\right) \in
  B^n_{\infty} \times_L \R^n$ with
  $\boldsymbol{\Psi}_{\epsilon}\left(\mathbf{x}_{\epsilon},\mathbf{y}_{\epsilon}\right)
  = \mathbf{z}$. Hence, it suffices to show that, for some $\epsilon >
  0$, $\left(\mathbf{x}_{\epsilon},\mathbf{y}_{\epsilon}\right) \in
  P_{\epsilon}$, which is equivalent to
  $\mathbf{I}_0\left(\Phi_{\epsilon}\left(\mathbf{x}_{\epsilon},\mathbf{y}_{\epsilon}\right)\right)
  \in 4|A|$, since $P_{\epsilon}$ is saturated with respect to
  $\Phi_{\epsilon}$. To see that this holds, we argue as follows. Choose a
  decreasing sequence $\epsilon_k$ converging to 0 and, for any $k$, set
  $\mathbf{c}_k =\left(c_{k,1},\ldots,c_{k,n}\right):= \Phi_{\epsilon_k}\left(\mathbf{x}_{\epsilon_k},
    \mathbf{y}_{\epsilon_k}\right)$, and
  $\boldsymbol{\mu}\left(\mathbf{z}\right)=:\left(a_1,\ldots,a_n\right)$. By
  assumption, we have that, for all $k$ and all $i =1,\ldots,n$,
  $I_{\epsilon_k}\left(c_{k,i}\right) = a_i$. Let $l > k$ and suppose
  that there exists $i =1,\ldots, n$ such that $c_{l,i} >
  c_{k,i}$. Then
  \begin{equation}
    \label{eq:14}
    a_i = I_{\epsilon_k}\left(c_{k,i}\right) <
    I_{\epsilon_l}\left(c_{k,i}\right) <
    I_{\epsilon_l}\left(c_{l,i}\right) < a_i,
  \end{equation}
  \noindent
  where the first inequality follows from the fact that if $\epsilon >
  \epsilon'$ and
  $c \geq \frac{\epsilon}{2}$, then $I_{\epsilon}(c) <
  I_{\epsilon'}(c)$ (see the proof of Proposition
  \ref{prop:decreasing_epsilon}), while the second follows from the
  fact that $I_{\epsilon_l}$ is a strictly increasing function (see
  the proof of Lemma \ref{lemma:action}). The inequalities \eqref{eq:14}
  yield a contradiction; thus, for all $l > k$ and all
  $i = 1,\ldots, n$, $c_{l,i} \leq
  c_{k,i}$. Together with the fact that, for all $k$, $\mathbf{c}_k
  \in \R^n_{\geq 0}$, this fact implies that the sequence
  $\left\{\mathbf{c}_k \right\}_k \subset \R^n_{\geq 0}$ is
  bounded. Therefore, without loss of generality, it may be assumed
  that $\mathbf{c}_k \to \mathbf{c}_{\infty} \in \R^n_{\geq 0}$. Hence,
  \begin{equation}
    \label{eq:15}
    \lim\limits_{k \to + \infty}\left( \lim\limits_{j \to +\infty}
      \mathbf{I}_{\epsilon_j}\left(\mathbf{c}_k\right)\right) =
    \lim\limits_{k \to + \infty}\mathbf{I}_0
    \left(\mathbf{c}_k\right) = \mathbf{I}_0\left(\mathbf{c}_{\infty}\right),
  \end{equation}
  \noindent
  where the first equality follows from Corollary
  \ref{cor:limit_general_case} and the second from continuity of
  $\mathbf{I}_0$. On the other hand,
  \begin{equation}
    \label{eq:16}
    \lim\limits_{k \to +\infty}
    \mathbf{I}_{\epsilon_k}\left(\mathbf{c}_k \right) = \boldsymbol{\mu}\left(\mathbf{z}\right),
  \end{equation}
  \noindent
  since, by definition, for all $k$,
  $\mathbf{I}_{\epsilon_k}\left(\mathbf{c}_k \right) =
  \boldsymbol{\mu}\left(\mathbf{z}\right)$. Comparing equations
  \eqref{eq:15} and \eqref{eq:16}, we obtain that
  $\mathbf{I}_0\left(\mathbf{c}_{\infty}\right) = \boldsymbol{\mu}\left(\mathbf{z}\right)$.
  Since
  $4|A| \subset \R^n_{\geq 0}$ is open and
  $\boldsymbol{\mu}(\mathbf{z}) \in 4|A|$, there exists a $\delta >0$ such that if
  $\mathbf{a}' \in \R_{\geq 0}^n$ and
  $\left\|\boldsymbol{\mu}\left(\mathbf{z}\right) - \mathbf{a}'
  \right\| < \delta$, then $\mathbf{a}' \in 4|A|$. Choose $k$
  sufficiently large so that $\left\|
    \mathbf{I}_0\left(\mathbf{c}_{\infty}\right) -
    \mathbf{I}_0\left(\mathbf{c}_{k}\right)\right\| <
  \frac{\delta}{2}$; this can be achieved since $\mathbf{I}_0$ is
  continuous and $\mathbf{c}_k \to \mathbf{c}_{\infty}$ as $k \to
  + \infty$. Hence,
  \begin{equation}
    \label{eq:17}
    \left\| \boldsymbol{\mu}\left(\mathbf{z}\right) -
      \mathbf{I}_0\left(\Phi_{\epsilon_k}\left(\mathbf{x}_{\epsilon_k},\mathbf{y}_{\epsilon_k}\right)\right)
    \right\| = \left\|\mathbf{I}_0\left(\mathbf{c}_{\infty}\right) -
      \mathbf{I}_0\left(\mathbf{c}_{k}\right) \right\| < \frac{\delta}{2};
  \end{equation}
  \noindent
  moreover, by definition of $\mathbf{I}_0$,
  $\mathbf{I}_0\left(\Phi_{\epsilon_k}\left(\mathbf{x}_{\epsilon_k},\mathbf{y}_{\epsilon_k}\right)\right)
  \in \R^n_{\geq 0}$. Thus
  $\mathbf{I}_0\left(\Phi_{\epsilon_k}\left(\mathbf{x}_{\epsilon_k},\mathbf{y}_{\epsilon_k}\right)\right)
  \in 4|A|$ as desired, which, unraveling the above argument,
  implies that $X_{4|A|} \subset \bigcup\limits_{\epsilon >0}
  \boldsymbol{\Psi}_{\epsilon}\left(P_{\epsilon}\right) \subset \bigcup\limits_{\epsilon >0}
  \boldsymbol{\Psi}_{\epsilon}\left(\mathrm{cl}\left(P_{\epsilon}\right)\right)$ and completes
  the proof.
  \end{proof}

  Claims \ref{claim:union_P} and \ref{claim:union_images} yield that
  the family of symplectic submanifolds
  $\left\{P_{\epsilon}\right\}_{\epsilon >0}$ satisfies property
  \ref{item:7}.

  \begin{claim}\label{claim:domain_nested}
    If $\epsilon_1 > \epsilon_2$, then
    $\mathrm{cl}\left(P_{\epsilon_1}\right) \subset \mathrm{cl}\left(P_{\epsilon_2}\right)$.
  \end{claim}
  \begin{proof}[Proof of Claim \ref{claim:domain_nested}]
    Fix $\epsilon_1 > \epsilon_2$. It suffices to show that
    $P_{\epsilon_1} \subset P_{\epsilon_2}$. Fix
    $\left(\mathbf{x},\mathbf{y}\right) \in P_{\epsilon_1}$. By
    definition,
    $\mathbf{I}_0\left(\Phi_{\epsilon_1}\left(\mathbf{x},\mathbf{y}\right)\right)
    \in 4|A|$, {\it i.e.}
    $$ \left(4\sqrt{y^2_1 + \frac{\epsilon_1}{1-x^2_1}}, \ldots,
      4\sqrt{y^2_n + \frac{\epsilon_1}{1-x^2_n}}\right) \in 4|A|. $$
    \noindent
    On the other hand, observe that, since $\epsilon_1 > \epsilon_2$,
    for all $i=1,\ldots, n$,
    $$ \sqrt{y^2_1 + \frac{\epsilon_1}{1-x^2_1}} > \sqrt{y^2_1 +
      \frac{\epsilon_2}{1-x^2_1}}. $$
    \noindent
    Since $4|A|$ satisfies property \eqref{eq:cond}, arguing as in
    the proof of Claim \ref{claim:union_P}, we obtain that
    $$ \left(4\sqrt{y^2_1 + \frac{\epsilon_2}{1-x^2_1}}, \ldots,
      4\sqrt{y^2_n + \frac{\epsilon_2}{1-x^2_n}}\right) \in 4|A|, $$
    \noindent
    which gives that $\mathbf{I}_0\left(\Phi_{\epsilon_2}\left(\mathbf{x},\mathbf{y}\right)\right)
    \in 4|A|$. By definition of $P_{\epsilon_2}$,
    $\left(\mathbf{x},\mathbf{y}\right) \in P_{\epsilon_2}$. Since
    $\left(\mathbf{x},\mathbf{y}\right) \in P_{\epsilon_1}$ is
    arbitrary, this shows that $P_{\epsilon_1} \subset P_{\epsilon_2}$
    as desired.
  \end{proof}

  \begin{claim}\label{claim:images_nested}
    If $\epsilon_1 > \epsilon_2$, then $\boldsymbol{\Psi}_{\epsilon_1}\left(\mathrm{cl}\left(P_{\epsilon_1}\right)\right)
    \subset \boldsymbol{\Psi}_{\epsilon_2}\left(\mathrm{cl}\left(P_{\epsilon_2}\right)\right)$.
  \end{claim}
  \begin{proof}[Proof of Claim \ref{claim:images_nested}]
    Fix $\epsilon_1 > \epsilon_2$. Since, for $i=1,2$,
    $\boldsymbol{\Psi}_{\epsilon_i}$ is a homeomorphism, it suffices to show that $\boldsymbol{\Psi}_{\epsilon_1}\left(P_{\epsilon_1}\right)
    \subset \boldsymbol{\Psi}_{\epsilon_2}\left(P_{\epsilon_2}\right)$.
    As, for $i=1,2$, the subset
    $\boldsymbol{\Psi}_{\epsilon_i}\left(P_{\epsilon_i}\right)$ is
    saturated with respect to $\boldsymbol{\mu}$, in order to prove
    the desired result it suffices to show that $\boldsymbol{\mu}\left(\boldsymbol{\Psi}_{\epsilon_1}\left(P_{\epsilon_1}\right)\right)
    \subset
    \boldsymbol{\mu}\left(\boldsymbol{\Psi}_{\epsilon_2}\left(P_{\epsilon_2}\right)\right)$,
    which is
    equivalent to
    $\mathbf{I}_{\epsilon_1}\left(\Phi_{\epsilon_1}\left(P_{\epsilon_1}\right)\right)
    \subset
    \mathbf{I}_{\epsilon_2}\left(\Phi_{\epsilon_2}\left(P_{\epsilon_2}\right)\right)$
    in light of Corollary \ref{cor:glob_sing_general}. Observe that,
    for $i=1,2$, $\Phi_{\epsilon_i}\left(P_{\epsilon_i}\right) =
    \mathbf{I}_0^{-1}\left(4|A|\right) \cap
    \left[\frac{\epsilon_i}{2},+\infty\right[^n$; thus, since
    $\epsilon_1 > \epsilon_2$,
    $\Phi_{\epsilon_1}\left(P_{\epsilon_1}\right) \subset \Phi_{\epsilon_2}\left(P_{\epsilon_2}\right)$. Let $\mathbf{c}
    = \left(c_1,\ldots,c_n\right) \in
    \Phi_{\epsilon_1}\left(P_{\epsilon_1}\right) \subset
    \Phi_{\epsilon_2}\left(P_{\epsilon_2}\right)$; the fact that
    $4|A|$ satisfies property \eqref{eq:cond} implies that
    \begin{equation}
      \label{eq:18}
      \left[\frac{\epsilon_2}{2},c_1\right] \times \ldots \times
      \left[\frac{\epsilon_2}{2},c_n\right] \subset \Phi_{\epsilon_2}\left(P_{\epsilon_2}\right).
    \end{equation}
    \noindent
    For, the condition $\mathbf{c} \in
    \Phi_{\epsilon_2}\left(P_{\epsilon_2}\right)$ implies that
    $\mathbf{I}_0\left(\mathbf{c}\right) = \left(4\sqrt{2c_1},\ldots,
      4\sqrt{2c_n}\right) \in 4|A|$. Since $4|A|$ satisfies
    property \eqref{eq:cond}, then
    $$ \left[0,4\sqrt{2c_1}\right] \times \ldots
    \left[0,4\sqrt{2c_n}\right] \subset 4|A|.$$
    \noindent
    Thus
    \begin{equation}
      \label{eq:19}
      \mathbf{I}_0^{-1}\left(\left[0,4\sqrt{2c_1}\right] \times \ldots
        \left[0,4\sqrt{2c_n}\right]\right) \subset
      \mathbf{I}_0^{-1}\left(4|A|\right);
    \end{equation}
    \noindent
    however, by definition of $\mathbf{I}_0$ (see Corollary \ref{cor:limit_general_case}),
    \begin{equation}
      \label{eq:20}
      \begin{split}
        \mathbf{I}_0^{-1}\left(\left[0,4\sqrt{2c_1}\right] \times \ldots
          \left[0,4\sqrt{2c_n}\right]\right) & = \left(I^{-1}_0
          \left(\left[0,4\sqrt{2c_1}\right] \right)\right) \times \ldots \times
        \left(I^{-1}_0 \left(\left[0,4\sqrt{2c_n}\right]\right)
        \right) \\
        &=
        \left[0,c_1\right] \times \ldots \left[0,c_n\right].
      \end{split}
    \end{equation}
    \noindent
    Equation \eqref{eq:18} follows by combining equations
    \eqref{eq:19} and \eqref{eq:20} with the equality $\Phi_{\epsilon_2}\left(P_{\epsilon_2}\right) =
    \mathbf{I}_0^{-1}\left(4|A|\right) \cap
    \left[\frac{\epsilon_2}{2},+\infty\right[^n$. Equation \eqref{eq:18} implies that
    \begin{equation}
      \label{eq:21}
      \left[0,I_{\epsilon_2}(c_1)\right] \times \ldots
      \left[0,I_{\epsilon_2}(c_n)\right] = \mathbf{I}_{\epsilon_2}\left(\left[\frac{\epsilon_2}{2},c_1\right] \times \ldots \times
        \left[\frac{\epsilon_2}{2},c_n\right]\right) \subset \mathbf{I}_{\epsilon_2}\left(\Phi_{\epsilon_2}\left(P_{\epsilon_2}\right)\right),
    \end{equation}
    \noindent
    where the first equality follows from the definition of
    $\mathbf{I}_{\epsilon_2}$ and properties of $I_{\epsilon_2}$ (see
    the proof of Lemma \ref{lemma:action}). Since $\epsilon_1 >
    \epsilon_2$, the proof of Proposition
    \ref{prop:decreasing_epsilon} gives that, for all $i=1,\ldots, n$,
    $ I_{\epsilon_1}(c_i) < I_{\epsilon_2}(c_i) $, which, together
    with equation \eqref{eq:21} gives that
    $\mathbf{I}_{\epsilon_1}\left(\mathbf{c}\right) =
    \left(I_{\epsilon_1}(c_1),\ldots,I_{\epsilon_1}(c_n)\right) \in
    \mathbf{I}_{\epsilon_2}\left(\Phi_{\epsilon_2}\left(P_{\epsilon_2}\right)\right)$. Since
    $\mathbf{c} \in \Phi_{\epsilon_1}\left(P_{\epsilon_1}\right)$ is
    arbitrary, the above argument shows that
    $\mathbf{I}_{\epsilon_1}\left(\Phi_{\epsilon_1}\left(P_{\epsilon_1}\right)\right)
    \subset
    \mathbf{I}_{\epsilon_2}\left(\Phi_{\epsilon_2}\left(P_{\epsilon_2}\right)\right)$
    as desired.
  \end{proof}

  Claims \ref{claim:domain_nested} and \ref{claim:images_nested} yield that
  the family of symplectic submanifolds
  $\left\{P_{\epsilon}\right\}_{\epsilon >0}$ satisfies property
  \ref{item:6}. This completes the proof.
\end{proof}

Lemma \ref{lemma:family} allows to prove Theorem \ref{thm:sympl}.

\begin{proof}[Proof of Theorem \ref{thm:sympl}]
  Fix a balanced region $A \subset \R^n$. The aim is to construct a symplectomorphism between
  $B^n_{\infty} \times_L A$ and the toric domain $X_{4|A|} $. Let
  $\left\{P_{\epsilon}\right\}_{\epsilon >0}$ be the family of
  symplectic submanifolds with compact closure of $B^n_{\infty} \times_L A$
  as in Lemma \ref{lemma:family}. Pick a decreasing
  sequence $\epsilon_k$ converging to $0$. By property \ref{item:6},
  for all $l > k$, $\mathrm{cl}\left(P_{\epsilon_l}\right) \subset \mathrm{cl}\left(P_{\epsilon_k}\right)$; moreover, combining properties
  \ref{item:7} and \ref{item:6} in Claim \ref{lemma:family}, $\bigcup\limits_{k \geq 1} \mathrm{cl}\left(P_{\epsilon_k}\right) =  B^n_{\infty}
  \times_L A$ and $\bigcup\limits_{k \geq 1}
  \boldsymbol{\Psi}_{\epsilon_k}\left(\mathrm{cl}\left(P_{\epsilon_k}\right)\right) =
  X_{4|A|}$.

  To construct the desired symplectomorphism we use an argument of \cite{mcduff_bl}
  which also appears in \cite[Proof of Theorem 3]{vinicius}. Fix $k
  \geq 2$. Observe that, for any $t \in
  \left[\epsilon_k,\epsilon_{k-1}\right]$,
  $$ \boldsymbol{\Psi}_{\epsilon_{k-1}}\left(\mathrm{cl}\left(P_{\epsilon_{k-1}}\right)\right)
  \subset \boldsymbol{\Psi}_{t}\left(\mathrm{cl}\left(P_t\right)\right) \subset
  \boldsymbol{\Psi}_{\epsilon_{k}}\left(\mathrm{cl}\left(P_{\epsilon_{k}}\right)\right), $$
  \noindent
  where the inclusions follow from property \ref{item:6} in Claim \ref{lemma:family}. Thus it
  is possible to consider an isotopy of symplectic embeddings
  $\boldsymbol{\Psi}^{-1}_t \circ \boldsymbol{\Psi}_{\epsilon_{k-1}} :
  \mathrm{cl}\left(P_{\epsilon_{k-1}}\right) \hookrightarrow \mathrm{cl}\left(P_{\epsilon_k}\right)$ for $t \in
  \left[\epsilon_k,\epsilon_{k-1}\right]$. Using the symplectic
  isotopy extension theorem (cf. \cite[Proposition 4]{auroux} and
  \cite{banyaga}), there exists an isotopy of symplectomorphisms
  $\chi_t : \mathrm{cl}\left(P_{\epsilon_k}\right) \to \mathrm{cl}\left(P_{\epsilon_k}\right)$ for $t \in
  \left[\epsilon_k,\epsilon_{k-1}\right]$ such that
  \begin{itemize}[leftmargin=*]
  \item $\chi_t|_{P_{\epsilon_{k-1}}} = \boldsymbol{\Psi}^{-1}_t \circ
    \boldsymbol{\Psi}_{\epsilon_{k-1}}$, and
  \item $\chi_t$ is the identity away from some neighborhood of $\mathrm{cl}\left(P_{\epsilon_{k-1}}\right)$.
  \end{itemize}
  The map
  $\widetilde{\boldsymbol{\Psi}}_{\epsilon_k} :=
  \boldsymbol{\Psi}_{\epsilon_k} \circ \chi_{\epsilon_k} :
  \mathrm{cl}\left(P_{\epsilon_k}\right) \hookrightarrow \R^{2n}$ is a
  symplectic embedding satisfying
  \begin{itemize}[leftmargin=*]
  \item
    $\widetilde{\boldsymbol{\Psi}}_{\epsilon_k}|_{\mathrm{cl}\left(P_{\epsilon_{k-1}}\right)}
    = \boldsymbol{\Psi}_{\epsilon_{k-1}}$, and
  \item $\widetilde{\boldsymbol{\Psi}}_{\epsilon_k}$ equals
    $\boldsymbol{\Psi}_{\epsilon_k}$ away from some neighborhood of $\mathrm{cl}\left(P_{\epsilon_{k-1}}\right)$.
  \end{itemize}
  Setting
  \begin{equation*}
    \boldsymbol{\Psi}\left(\mathbf{x},\mathbf{y}\right):=
    \begin{cases}
      \boldsymbol{\Psi}_{\epsilon_1}\left(\mathbf{x},\mathbf{y}\right)
      & \text{ if } \left(\mathbf{x},\mathbf{y}\right) \in
      \mathrm{cl}\left(P_{\epsilon_1}\right), \\
      \widetilde{\boldsymbol{\Psi}}_{\epsilon_k}\left(\mathbf{x},\mathbf{y}\right)
      & \text{ if } \left(\mathbf{x},\mathbf{y}\right) \in
      \mathrm{cl}\left(P_{\epsilon_k}\right) \smallsetminus \mathrm{cl}\left(P_{\epsilon_{k-1}}\right),
    \end{cases}
  \end{equation*}
  \noindent
  we obtain a well-defined map $\boldsymbol{\Psi} : \bigcup\limits_{k \geq 1} \mathrm{cl}\left(P_{\epsilon_k}\right) =  B^n_{\infty}
  \times_L A \to \R^{2n}$. The above properties imply
  that $\Psi$ is a symplectic embedding of $B^n_{\infty}
  \times_L A$ into $\R^{2n}$ whose
  image equals $\bigcup\limits_{k \geq 1}
  \boldsymbol{\Psi}_{\epsilon_k}\left(\mathrm{cl}\left(P_{\epsilon_k}\right)\right) =
  X_{4|A|}$ as desired.
\end{proof}

\bibliography{bib}{}
\bibliographystyle{plain}

\end{document}